\documentclass[a4paper,11pt]{article}
\usepackage{amsfonts}
\usepackage[xdvi]{graphicx}
\usepackage[dvips]{color}
\usepackage{amsmath,amssymb,amsthm}
\usepackage[all]{xy}
\usepackage{bm}
\usepackage{xcolor}

\newcommand{\R}{{\mathbb R}}
\newcommand{\C}{{\mathbb C}}
\newcommand{\N}{{\mathbb N}}
\newcommand{\Z}{{\mathbb Z}}
\newcommand{\dd}{{\rm d}}

\makeatletter

\@addtoreset{equation}{section}

\theoremstyle{definition}
\newtheorem{theorem}{Theorem}
\newtheorem{definition}[theorem]{Definition}
\newtheorem{corollary}[theorem]{Corollary}
\newtheorem{lemma}[theorem]{Lemma}
\newtheorem{proposition}[theorem]{Proposition}
\newtheorem{remark}[theorem]{Remark}
\numberwithin{theorem}{section}

\title{Generalized eigenvalues of the Perron-Frobenius operators of 
symbolic dynamical systems}
\date{\today}

\author{Hayato Chiba\thanks{Advanced Institute for Materials Research,
    Tohoku  University, Sendai, 980-8557, Japan, {\tt hchiba@tohoku.ac.jp}}\,,
  \quad 
Masahiro Ikeda\thanks{Center for Advanced Intelligence Project, RIKEN, Tokyo, 103-0027, Japan / Department of Mathematics, Keio University, Yokohama, 223-8522, Japan, 
{\tt masahiro.ikeda@riken.jp}} 
  \quad and\quad
Isao Ishikawa\thanks{Center for Data Science, Ehime University, Matsuyama, 790-8577, Japan / Center for Advanced Intelligence Project, RIKEN, Tokyo, 103-0027, Japan
{\tt ishikawa.isao.zx@ehime-u.ac.jp}}   
}
\begin{document}
\maketitle

\begin{abstract}
The generalized spectral theory is an effective approach to analyze a linear operator on a Hilbert space $\mathcal{H}$ with a continuous spectrum.
The generalized spectrum is computed via analytic continuations of the resolvent operators using a dense locally convex subspace $X$ of $\mathcal{H}$ and its dual space $X'$. The three topological spaces $X \subset \mathcal{H} \subset X'$ is called the rigged Hilbert space or the Gelfand triplet.
In this paper, the generalized spectra of the Perron-Frobenius operators of the one-sided and two-sided shifts of finite types (symbolic dynamical systems) are determined.
A one-sided subshift of finite type which is conjugate to the multiplication with the golden ration on $[0,1]$ modulo $1$ is also considered. 
A new construction of the Gelfand triplet for the generalized spectrum of symbolic dynamical systems is proposed by means of an algebraic procedure.
The asymptotic formula of the iteration of Perron-Frobenius operators is also given. The iteration converges to the mixing state whose rate of convergence is determined by the generalized spectrum.
\end{abstract}

\section{Introduction}

The spectrum of a linear operator on a topological vector space is 
an important objective in any area of mathematics.
If an underlying space is a finite dimensional space, the spectrum set consists of eigenvalues, while on an infinite dimensional space, the spectrum set may include a continuous and a residual spectra.
From eigenvalues, a lot of information of a linear operator is obtained by using, such as eigenvectors, projections to eigenspaces and the residue theorem.
However, a continuous and a residual spectra are far from tractable because how to treat them depends on the problem at hand and a standard method has not been established yet.

The generalized spectral theory based on a rigged Hilbert space is one of the methods to handle problems related to a continuous and a residual spectra.

Let $T$ be a linear operator on a Hilbert space $\mathcal{H}$ and $(\lambda -T)^{-1}$ its resolvent operator.
The resolvent set $\rho (T) \subset \C$ is defined such that $(\lambda -T)^{-1}$ is a continuous operator on $\mathcal{H}$ when $\lambda \in \rho (T)$.
The spectrum set is defined as its complement: $\sigma (T) = \C\backslash \rho(T)$.
It is well known that the resolvent $(\lambda -T)^{-1}$ is holomorphic with respect to $\lambda $
on the resolvent set $\rho(T)$, while is singular on $\sigma (T)$.
A basic idea of the generalized spectral theory is that by controlling the topology of the underlying vector space well,
a continuous and a residual spectra disappear, and generalized eigenvalues may appear instead, which play a similar role to the usual eigenvalues.

Specifically, let $X$ be a dense subspace of $\mathcal{H}$ whose topology is stronger than that of $\mathcal{H}$,
and $X'$ its dual space. Then, the three topological vector spaces
\begin{eqnarray*}
X \subset \mathcal{H} \subset X'
\end{eqnarray*}
is called the rigged Hilbert space or the Gelfand triplet.
In the generalized spectral theory, the domain of the resolvent $(\lambda -T)^{-1}$ is restricted to $X$ and the range is 
extended to $X'$.
By regarding the resolvent as the operator from $X$ into $X'$, under a suitable assumption,
it has an analytic continuation from the set $\rho(T)$ into the spectrum set $\sigma (T)$.
Then, the region on which $(\lambda -T)^{-1}$ is holomorphic becomes a nontrivial Riemann surface.
The analytic continuation of the resolvent is called the generalized resolvent.
The set of singularities of the analytic continuation of $(\lambda -T)^{-1}$ on the Riemann surface is called the generalized spectrum set.
A point of the generalized spectrum set is called the generalized eigenvalue if a corresponding eigenvector exists in the dual space $X'$.
Since the generalized eigenvalue lies on a Riemann surface and the associated eigenvector 
is not an element of $\mathcal{H}$, 
we can investigate phenomena that cannot be captured by the usual spectrum set based on the 
Hilbert space theory.
See Section 2 for the detail. 

Generalized eigenvalues have been well studied in quantum mechanics (in this area, a generalized eigenvalue is called a resonance or a resonance pole) \cite{AAD, Bo86, BG89, Chi3, Fro97, Hit, Sim, Zwo}.
It was defined as a pole of a resolvent, Green function, a scattering matrix, and so on.
Typically the imaginary part of a resonance describes the decay rate of waves.

A mathematical formulation of the generalized spectrum based on the Gelfand triplet was developed by \cite{BG89, Chi2, Gel2, Gel} and references therein. 

Applications to chaos in dynamical systems was first suggested by Ruelle \cite{Ru86}.
After his work, many applications to dynamical systems have been studied.
In Chiba \cite{Chi1, Chi5}, the generalized spectral theory is applied to the bifurcation of the Kuramoto model, which is the well known mathematical model for synchronization phenomena.
Since this system has the continuous spectrum on the imaginary axis and the standard bifurcation theory is not applicable, the analytic continuation of the resolvent operator from the right half plane to the left half plane was constructed and it is shown that the generalized eigenvalues on the Riemann surface determine the stability and bifurcation of synchronized states.
In \cite{Chi4}, a similar strategy was applied to a neuronal network to show the existence and stability of the gamma wave, one of the brain wave related to cognition and attention.

Applications to discrete dynamical systems (iteration of a transformation) are also attractive topics studied by many authors.

Let $T$ be a measure preserving continuous transformation of a probability space $(\Sigma, \mathcal{B}, \mu)$,
where $\Sigma$ is a compact metric space, $\mathcal{B}$ is a Borel $\sigma $-algebra on $\Sigma$ and $\mu$ is an invariant measure of $T$.
The Koopman operator (composition operator) $U_T$ of $T$ is defined by $(U_T f)(x) = f(Tx)$
for a measurable function $f : \Sigma \to \C$.
It satisfies
\begin{eqnarray*}
\int_{\Sigma} (U_Tf)(x) \overline{(U_Tg)(x)} \dd\mu =  \int_{\Sigma} f(x) \overline{g(x)} \dd\mu
\end{eqnarray*}
for $f,g \in L^2 (\Sigma, \mu)$; $U_T$ is an isometry on the Lebesgue space $L^2 (\Sigma, \mu)$.
The Perron-Frobenius operator (transfer operator) $V_T = U_T^*$ is defined as the adjoint of 
$U_T$ as $(V_Tf, g)_{L^2} = (f, U_Tg)_{L^2}$.
If $T$ is a homeomorphism on $\Sigma$, $U_T$ is unitary and $V_T = U_T^{-1} = U_{T^{-1}}$.
Even if the space $\Sigma$, on which a discrete dynamical system is defined, is a finite dimensional manifold (actually $\Sigma = [0,1]$ interval for most studies using the generalized spectral theory),
the Perron-Frobenius/Koopman operators are linear operators on an infinite dimensional space $L^2 (\Sigma, \mu)$ and they may have continuous spectra.
Indeed, it is known that the transformation $T$ is weak mixing if and only if the Koopman operator has a continuous spectrum~\cite{CFS, Wal}.

In 90's, there are many works on discrete dynamical systems on the $[0,1]$ interval from a viewpoint of the Gelfand triplet~\cite{ASS99, AT93, SATB96}, in particular, for the R\'{e}yni map~\cite{ADKM97, ASS99, AT92, HS92, HD94, Gas92} and the baker map~\cite{AT92a, HS92, HD94, Fox97}.

It is known that many dynamical systems on $[0,1]$ are topologically (semi-)conjugate to symbolic dynamical systems.
In this paper, spectral properties and the asymptotic behavior of the Perron-Frobenius operators of symbolic dynamical systems are studied by means of the generalized spectral theory.
Let $[\beta] := \{0, 1,\cdots ,\beta-1\}$ be a finite set of $\beta$ symbols equipped with a discrete topology.
The sets $\Sigma = [\beta]^{\Z}$ and $\Sigma_+ = [\beta]^{\N}$ of sequences from $[\beta]$ are metrizable spaces with the product topology.
The two-sided full shift of finite type $(\Sigma, S)$ is a pair of the space
\begin{eqnarray*}
\Sigma := \{ \omega = (\cdots , \omega _{-1}, \omega _0.\, \omega _1, \cdots )\, | \, \omega _i\in [\beta]\}
\end{eqnarray*}
and the left shift operator $(S\omega )_j = \omega _{j+1}$.
Let $A = (a_{ij})_{i,j}$ be an $n\times n$ adjacency matrix with entries in $\{ -1,1\}$.
The two-sided subshift of finite type $(\Sigma^A, S)$ is a space
\begin{eqnarray*}
\Sigma^A := \{ \omega = (\cdots , \omega _{-1}, \omega _0.\, \omega _1, \cdots )\, | \, 
   \omega _i\in [\beta],\, a_{\omega _j \omega _{j+1}} = 1 \}
\end{eqnarray*}
and the left shift operator $S$ on it.
The one-sided left shift of finite type is defined in a similar manner on
\begin{eqnarray*}
\Sigma_+ := \{ \omega = (\omega _1, \omega_2,  \cdots )\, | \, \omega _i\in [\beta]\},
\end{eqnarray*}
or 
\begin{eqnarray*}
\Sigma^A_+ := \{ \omega = (\omega _1, \omega_2,  \cdots )\, | \, 
   \omega _i\in [\beta],\, a_{\omega _j \omega _{j+1}} = 1 \}.
\end{eqnarray*}
Since the Perron-Frobenius operator $V_S$ of the shift map $S$ is isometry, the spectrum set is included in the unit disk. In fact, the Perron-Frobenius operators considered in this paper have the continuous spectra on the unit circle. In order to apply the generalized spectral theory, a topological vector space $X$ will be carefully chosen so that $X \subset  L^2 (\Sigma, \mu) \subset X'$ becomes a suitable Gelfand triplet.
Then, it is shown that the resolvent operator $(\lambda - V_S)^{-1}$ has an analytic continuation from the outside of the unit circle to the inside as an operator from $X$ into $X'$, although $(\lambda - V_S)^{-1}$ is singular on the unit circle as an operator on $L^2 (\Sigma, \mu)$. We will find infinitely many generalized eigenvalues inside the unit circle and give the algorithm to calculate them for the one-sided shift map including a nontrivial adjacency matrix in Section 3, and for the two-sided shift map in Section 4.
They describe the transient of the dynamics of the iteration of the Perron-Frobenius operators, in particular, the rate of convergence to mixing state.
Section 2 is devoted to a brief review of the generalized spectral theory according to \cite{Chi2}.


\section{The generalized spectral theory}

Let $X$ be a locally convex Hausdorff topological vector space over $\C$ and $X'$ its dual space.
$X'$ is the set of continuous anti-linear functionals on $X$.
For $\mu \in X'$ and $f \in X$, $\mu (f )$ is denoted by
$\langle \mu \,|\, f \rangle$.
For any $a,b \in \C,\, f, g \in X$ and $\mu, \xi \in X'$, the equalities
\begin{eqnarray*}
& & \langle \mu \,|\,  a f + bg\rangle 
   = \overline{a} \langle \mu \,|\,  f \rangle + \overline{b} \langle \mu \,|\, g \rangle, \\
& & \langle a\mu + b\xi \,|\, f \rangle
   = a \langle \mu \,|\, f \rangle + b \langle \xi \,|\, f \rangle,
\end{eqnarray*}
hold. 
The dual space $X'$ is equipped with the weak dual topology (weak * topology) ;
A sequence $\{ \mu_j\} \subset X'$ is said to be weakly convergent to $\mu \in X'$
if $\langle \mu_j \,|\, f \rangle \to \langle \mu  \,|\, f \rangle$ as $j\rightarrow\infty$ for each $f \in X$.

Let $\mathcal{H}$ be a Hilbert space with the inner product $(\cdot\, , \, \cdot)$ such that $X$ is a dense subspace of
$\mathcal{H}$.
Since a Hilbert space is isomorphic to its dual space, we obtain $\mathcal{H} \subset X'$ through $\mathcal{H} \simeq \mathcal{H}'$.
The isomorphism $\mathcal{H} \simeq \mathcal{H}'$ is defined so that $\langle g \,|\, f \rangle = (g, f)$
when $g\in \mathcal{H}$; that is, the dual pair is compatible with the inner product.
\\

\begin{definition}
If a locally convex Hausdorff topological vector space $X$ is a dense subspace of 
a Hilbert space $\mathcal{H}$ and the topology of $X$ is stronger than that of $\mathcal{H}$, the triplet
\begin{equation}
X \subset \mathcal{H} \subset X'
\end{equation}
is called the \textit{rigged Hilbert space} or the \textit{Gelfand triplet}.
\end{definition}

In this case, $\mathcal{H}$ is a dense subspace of the dual space
$X'$ and the topology of $\mathcal{H}$ is stronger than that of $X'$.

Let $T$ be a linear operator densely defined on $\mathcal{H}$.
The (Hilbert) adjoint $T^*$ of $T$ is defined through $(Tf, g) = (f , T^* g)$ as usual.
If $T^*$ preserves $X$ and is continuous on $X$, the dual operator $T^\times : X' \to X'$ of $T^*|_X$ defined through
\begin{eqnarray*}
\langle T^\times \mu \,|\, f \rangle = \langle \mu  \,|\, T^* f \rangle,\quad
f \in X,\, \mu\in X'
\end{eqnarray*}
is also continuous on $X'$ with respect to the weak dual topology.
We can show the equality $T^\times g = Tg$ for any $g \in X$, which implies that $T^\times $ is an extension of $T$.
To define a generalized spectrum set of $T$, we further assume that $X$ is a quasi-complete barreled space
(see Tr\'{e}ves \cite[Definition 34.1, Definition 33.1]{Tre} for the definitions), which is used to justify an integration of 
$X'$-valued complex functions \cite{Chi2}.
Any Fr\'{e}chet spaces, Banach spaces and Montel spaces satisfy this condition. 

Let $\rho (T)$ be a resolvent set of $T$; the resolvent operator $R_\lambda =(\lambda -T)^{-1}$
is a continuous operator on $\mathcal{H}$ when $\lambda \in \rho (T)$ 
and is holomorphic on $\rho (T)$.
The set $\sigma (T): = \C \backslash \rho (T)$ is called the spectrum set of $T$,
on which the resolvent operator is not 
defined as a continuous operator on $\mathcal{H}$.
For typical Perron-Frobenius operators considered in this paper, both of point spectra
(eigenvalues) and the continuous spectrum lie on the unit disk.
For an isolated eigenvalue $\lambda $, the Riesz projection to its eigenspace is given by
\begin{eqnarray*}
\Pi_\lambda [\phi] = \frac{1}{2\pi i} \oint_C (z-T)^{-1}\phi dz,\quad \phi \in \mathcal{H},
\end{eqnarray*}
where $C$ is a positively oriented closed curve, which encircles an eigenvalue $\lambda$
  but no other spectrum.
  The eigenspace of $\lambda$ is given by the image of $\Pi_\lambda$ as long as it is a finite dimensional space.
  
  However, the continuous spectrum has no corresponding eigenvectors that makes many problems difficult. To overcome this difficulty caused by the continuous spectrum, the generalized spectral theory has been developed.
For a suitable choice of a Gelfand triplet (depending on the problem at hand),
it is shown that if the resolvent $(\lambda -T)^{-1}$ is regarded as an operator from $X$ into $X'$
(i.e. the domain is restricted to $X$ and the range is extended to $X'$),
it has an analytic continuation beyond the continuous spectrum, which is called the  generalized resolvent operator.
The region, on which the resolvent $(\lambda -T)^{-1}$ is holomorphic, becomes a nontrivial Riemann surface and the continuous spectrum becomes its branch cut
\footnote{If the continuous spectrum is a curve, the edge point could be a branch point of the Riemann surface. For example, the continuous spectrum of the Laplace operator on $\R^m$ is the negative real axis and the edge point (the origin) becomes a branch point of the Riemann surface of the resolvent of type $z^{1/2}$ (resp. logarithmic) when $m$ is odd (resp. even) \cite{Chi3}.}. 
If the generalized resolvent has a singularity on the Riemann surface,
it is called the generalized spectrum.
It is also called the generalized eigenvalue, which plays a similar role to a usual eigenvalue, if the associated eigenvector exists in the dual space $X'$.
The Riesz projection for the isolated generalized eigenvalue is same as above but it is an operator from $X$ into $X'$ because so is the generalized resolvent.
We define the generalized eigenspace, which is a subspace of $X'$, as the image of the Riesz projection and the generalized eigenvectors are the elements of the generalized eigenspace.
It is known that there exists an element $\mu$ in the image of the generalized Riesz projection satisfying (\cite{Chi2} Theorem 3.5)
\begin{eqnarray*}
T^\times \mu = \lambda \mu, \quad \mu\in X'.
\end{eqnarray*}
That is, a generalized eigenvector is a true eigenvector of the dual operator.
Remark that not all the eigenvalues of the dual operator $T^\times : X' \to X'$ is a 
generalized eigenvalue in the sense of the generalized spectral theory.
Usually, the spectrum set of $T^\times : X' \to X'$ is too large and eigenspaces
are often infinite dimensional, so that it is not suitable to investigate the properties of $T$.
On the other hand, the number of the generalized eigenvalues is at most countable and their generalized eigenspaces are finite dimensional under a mild assumption (\cite{Chi2} Theorem 3.19).

For the Perron-Frobenius operator $V_T$ considered in this paper, 
its generalized resolvent $(\lambda -V_T)^{-1}: X \to X'$ has an analytic continuation from the outside 
of the unit circle to the inside beyond the continuous spectrum on the unit circle.
Then, we will find infinitely many generalized eigenvalues inside the unit circle. 


\section{Generalized eigenvalues of one-sided shifts} \label{sec: generalized spectrum of one-sided shifts}
In this section, we obtain the generalized eigenvalues of the Perron-Frobenius operator
of the left shift operator $(S\omega )_j := \omega _{j+1}$ on the one-sided shift space
\begin{equation}
\label{eq: def of Sigma_A}
\Sigma_+^A := \{ \omega = (\omega _1, \omega _2, \cdots )\, | \, 
   \omega _i\in [\beta], a_{\omega_i \omega _{i+1}} = 1 \},
\end{equation}
with the $\beta\times \beta$ adjacency matrix $A = (a_{ij})_{i,j}$ and a suitable invariant measure $\mu$ of $S$.
The Koopman operator $U_S$ is defined by
\begin{equation}
(U_S f)(\omega _1, \omega _2, \cdots ) := f(\omega _2, \omega _3, \cdots )
\end{equation}
and the Perron-Frobenius operator $V_S$ is given by its adjoint on $L^2 (\Sigma_+^A, \mu)$.

In Section \ref{subsec: full 2 shift} and \ref{subsec: full beta shift} the Bernoulli shift is studied; that is, $a_{ij} = 1$ for $\forall i,j$
and an invariant measure $\mu$ is a uniformly distributed Bernoulli measure in Section \ref{subsec: full 2 shift} and 
non-uniformly case in Section \ref{subsec: full beta shift}.
In Section \ref{subsec: partial shift}, the subshift defined by the adjacency matrix
\begin{equation}
A = \left(
\begin{array}{@{\,}cc@{\,}}
1 & 1 \\
1 & 0
\end{array}
\right)
\end{equation}
will be considered.
All cases are topological semi-conjugate to certain piecewise linear dynamical systems on the interval $[0,1]$.
A combination of these three results suggests how to obtain generalized eigenvalues for a general one-sided subshift of finite type.
\subsection{The one-sided full 2-shift}
\label{subsec: full 2 shift}
In this subsection, we describe an idea and the detailed calculation of the generalized eigenvalues of the one-sided shift.
Here, we mainly deal with the one-sided full 2-shift, namely, $\Sigma_+ =\{0,1\}^\N$ case. We will treat a general case in subsequent sections.

First, we fix several notations.
Let $\mu_0$ be a measure on the set $\{0,1\}$ defined by $\mu_0(\{0\})=\mu_0(\{1\})=1/2$, and define a measure $\mu_+:= \mu_0^{\otimes \N}$ on $\Sigma_+$ as a product measure of $\mu_0$'s.
We note that $\mu_+$ is an invariant measure with respect to $S$, and thus $U_{S}$ is an isometric linear operator on $L^2(\Sigma_+,\mu_+)$.
For $x=0,1$ and $\omega=(\omega_1,\omega_2,\dots)\in\Sigma_+$, we denote by $x*\omega\in\Sigma_+$ the sequence $(x,\omega_1,\omega_2,\dots)\in\Sigma_+$.

We remark that the one-sided full 2-shift space $\Sigma_+$ is also regarded as a compact topological group.
Actually, $\Sigma_+$ is a compact group defined by the product of the groups $\Z/2\Z$'s. 
Then, the normalized Haar measure on this group is identical with the measure $\mu_+$ defined above.  
For $k\in\N$, we define a character on $\Sigma_+$, which is a continuous group homomorphism from $\Sigma_+$ to $\C$, by
\[\rho_k:\Sigma_+\longrightarrow \{1,-1\};\omega=(\omega_1,\omega_2,\dots)\mapsto (-1)^{\omega_k}. \]
For $n\in\Z_{\ge0} $, let $n=\sum_{i=0}^\infty s_i2^{i}$ be the 2-adic expansion. 
Then, the {\em $n$-th Walsh function $W_n$} is defined by 
\[W_n(\omega) := \prod_{k=0}^\infty\rho_{k+1}(\omega)^{s_k}.\]
The set of Walsh functions $\{W_n\}_{n=0}^\infty$ generates the whole characters of $\Sigma_+$. 
Thus, by the representation theory of compact topological groups, $\{W_n\}_{n=0}^\infty$ is an orthonormal basis of $L^2(\Sigma_+,\mu_+)$.

First, we give several basic properties of $U_S$ and $V_{S}$ on $L^2(\Sigma_+,\mu_+)$.  
\begin{proposition}
\label{prop: 193308242020}
\begin{enumerate}
    \item \label{PF op of one full shift}For $f\in L^2(\Sigma_+,\mu_+)$, we have 
    \begin{align}
        (V_{S}f)(\omega)=2^{-1}f(0*\omega)+2^{-1}f(1*\omega).
    \end{align}
    \item \label{PF op and Walsh function}For $n\in\Z_{\ge0}$, we have
    \begin{align}
    U_{S}W_n&=W_{2n},\\
    V_{S}W_n&=
    \begin{cases}
        W_{n/2} & \text{ if }n\in 2\Z,\\
        0 & \text{otherwise}.
    \end{cases}
    \end{align}
    \item \label{spc of PF op}The spectrum of $V_{S}$ is described as follows:
    \begin{align*}
        \text{point spectrum: }& \{z \in \C \,|\, |z|<1\}\cup \{1\},\\
        \text{continuous spectrum: }& \{z \in \C \,|\, |z|=1, z\neq 1\}.
    \end{align*}
    The eigenspace of the eigenvalue $z = 1$ is one dimensional space consisting of constant functions. Otherwise, the eigenspaces of the point spectrums $z$ are infinite dimensional.
\end{enumerate}
\end{proposition}
\begin{proof}
Regarding (\ref{PF op of one full shift}), since $U_S$ is the adjoint of $V_S$,
\begin{align*}
    \int_{\Sigma_+} f(\omega) \overline{V_Sg(\omega)} \dd\mu_+(\omega)
    &= \int_{\Sigma_+} U_{S} f(\omega) \overline{g(\omega)} \dd\mu_+(\omega)\\
    &= \int_{\Sigma_+} f(\omega_2,\dots) \overline{g(\omega_1,\omega_2,\dots)}\dd\mu_+(\omega).
\end{align*}
Then, we have
\begin{align*}
    \int_{\Sigma_+} f(\omega)\overline{ V_Sg(\omega)} \dd\mu_+(\omega)
    &= \int_{\{0,1\}} \int_{\Sigma_+} f(\widetilde{\omega})\overline{g(x*\widetilde{\omega})} \dd\mu_+(\widetilde{\omega})\dd\mu_0(x)\\
    &=\int_{\Sigma_+}f(\omega)\overline{\left(2^{-1}g(0*\omega)+2^{-1}g(1*\omega)\right) }\dd\mu_+(\omega).
\end{align*}
As for (\ref{PF op and Walsh function}), it follows from a direct calculation with (\ref{PF op of one full shift}).
Now, we prove (\ref{spc of PF op}). First, we consider the point spectrum. Let $f\in L^2(\Sigma_+,\mu_+)$ be an eigenfunction of $V_{S}$ such that $V_{S}f=\lambda f$ for some $\lambda \in \C$. 
Since $\|V_{S}\|=1$, we may assume $|\lambda|\le1$.
Let $f=\sum_{n=0}^\infty c_nW_n$ be the orthonormal decomposition with the Walsh functions.
By substituting it into $V_Sf=\lambda f$, and comparison of coefficients for all even positive integers $n\in 2 \Z_{\ge0}$ with the aid of the statement (\ref{PF op and Walsh function}), we have 
\[c_n=
\lambda c_{n/2}.\]
Since $\sum_n|c_n|^2<\infty$, we conclude $\lambda=1$ and $c_n=0$ for all $n>0$, or $|\lambda|<1$.
The former case implies that the eigenspace of the eigenvalue 1 is 1-dimensional, and the eigenfunctions are constant functions.
Regarding the latter case, since $f$ is described as
\[f=\sum_{\substack{n\in\N \\ n\not\equiv0~{\rm mod}~2}}c_n\left(\sum_{r=0}^\infty \lambda^{r}W_{2^rn}\right),\]
for arbitrary constants $c_n$ ($n$: odd), where we define $0^0=1$ in the case of $\lambda=0$, the eigenspace of the eigenvalue $\lambda$ with $|\lambda|<1$ is an infinite dimensional space, which consists of eigenfunctions $\sum_{r\ge0} \lambda^{r+1}W_{2^rn}$ for odd integers $n$.
Let us consider the continuous spectrum of $V_{S}$.
Let $z\in\C\setminus\{1\}$ such that $|z|=1$. Note that the above argument implies that $z$ cannot be a point spectrum. For $n>0$, we define 
\[f_n:=\frac{1}{\sqrt{n}}\sum_{k=0}^{n-1} z^kW_{2^k}. \]
Then, we have $\| f_n\|=1$ and the statement (2) shows $(z-V_S)f_n=z^{n}W_{2^{n-1}}/\sqrt{n}$. 
Thus $(z-V_S)f_n\rightarrow 0$ as $n\rightarrow \infty$, namely, $z$ is a continuous spectrum of $V_S$.
\end{proof}

We define the function $h:\Sigma_+\rightarrow [0,1]$ by
\begin{align}
\label{eq: def of h}
   h(\omega)=\sum_{i=1}^\infty \omega_i2^{-i}.
\end{align}
We note that $h$ is a continuous and surjective map, 
and induces a homeomorphism outside the countable subset of periodic sequences, namely, $h$ is a homeomorphism from $\Sigma_+\setminus \{\omega=(\omega_i)_i : \omega_i=\omega_{i+r} \text{ for some $r$}\}$ onto $[0,1]\setminus \Z[2^{-1}]$, where $\Z[2^{-1}]=\{n2^{-m} \,|\, m,n\in\Z\}$. 
In particular,  the pull-back $h^*$ induces an isomorphism between the Hilbert spaces:
\begin{align}h^*: L^2([0,1],\dd x) \cong L^2(\Sigma_+,\mu_+); x\mapsto h.
\label{eq:isom of L^2}
\end{align}
Remark that this isomorphism gives compatibility between  Perron-Frobenius operators associated with the shift map and the 2-adic R\'enyi map. More precisely, let $T$ be the 2-adic R\'enyi map, which is a map on $[0,1]$ defined by
\begin{align}
    T(x)=
    \begin{cases}
    2x & \text{ if }x\in[0,1/2],\\
    2x-1 & \text{ if }x\in(1/2,1].
    \end{cases}
\end{align}
We have the following commutative diagram:
\[
\xymatrix{
L^2([0,1],\dd{x}) \ar[r]^{h^*}\ar[d]^{V_T} & L^2(\Sigma_+,\mu_+)\ar[d]^{V_S}\\
L^2([0,1],\dd{x}) \ar[r]^{h^*} & L^2(\Sigma_+,\mu_+), \\
}
\]
namely, 
\begin{align}
    h^*V_T=V_Sh^*.
\end{align}
As in the following proposition, the action of $V_S$ preserves the degree of polynomials of variable $h \in L^2(\Sigma_+, \mu_+)$:
\begin{proposition}
\label{prop: behavior of h via PF op}
We have
\begin{align*}
    h(0*\omega) &= 2^{-1} h(\omega), \\
    h(1*\omega) &= 2^{-1} + 2^{-1}h(\omega).
\end{align*}
In particular,
\[V_{S}[h^n] = 2^{-n}h^n + q(h), \]
where $q(h)\in \C[h]$ is a polynomial of variable $h$ of degree smaller than $n$.
\end{proposition}
\begin{proof}
It follows from the definition of $h$.
\end{proof}

\paragraph{Calculation of the generalized eigenvalues.}
We compute the generalized eigenvalues for $V_{S}$.
First, we specify a test space $X$ to determine a rigged Hilbert space.
Let 
\begin{align}
    X_n:=\sum_{i=0}^n \C h^i \label{eq: def of X_n}
\end{align}
be a $\C$-linear subspace of dimension $n+1$ in  $L^2(\Sigma_+,\mu_+)$.
We equip $X_n$ with the usual topology of $\C^{n+1}$.
Then, we define the space
\begin{align}
    X:=\lim_{\underset{n}{\longrightarrow}}X_n =\C[h], \label{eq: def of X}
\end{align}
equipped with the inductive limit topology.
\begin{proposition}
The space $X$ is a dense subspace of $L^2(\Sigma_+,\mu_+)$ and has a stronger topology than $L^2(\Sigma_+,\mu_+)$.
\end{proposition}
\begin{proof}
Via the isomorphism (\ref{eq:isom of L^2}), the space $X$ corresponds to the space of polynomials in $L^2([0,1],\dd x)$, which is a dense subspace due to the Stone-Weierstrass theorem.
Since the inclusion $X_n \hookrightarrow L^2(\Sigma_+, \mu_+)$ is continuous for any $n$, thus the inclusion $X \hookrightarrow L^2(\Sigma_+, \mu_+)$ is also continuous according to the definition of the topology of the inductive limit, which is the strongest topology such that any $X_n \hookrightarrow X$ is continuous.
In particular, $X$ has a stronger topology than $L^2(\Sigma_+, \mu_+)$.
\end{proof}
We note that $X$ is a Montel space.  Hence, the rigged Hilbert space $X\subset L^2(\Sigma_+,\mu_+)\subset X'$ is well-defined. 
Then, $X$ and $X_n$ satisfy the following proposition:
\begin{proposition}
\label{prop:2020-08-18-1323}
For $n\in\Z_{\ge0}$, we have
\[V_{S}X_n\subset X_n.\]
Moreover, the representation matrix of $V_{S}|_{X_n}$ associated to the basis $1,h,\dots,h^n$ is an upper triangular matrix whose diagonal components are $1,2^{-1},\dots,2^{-n}$.
In particular, $V_{S}$ induces a continuous linear operator on $X$.
\end{proposition}
\begin{proof}
It follows from Proposition \ref{prop: behavior of h via PF op}.
\end{proof}
\begin{corollary}
\label{cor: 202008241553}
For each $n\in\Z_{\ge0}$, the eigenspace of $V_S|_X$ in $X$ corresponding to the eigenvalue $2^{-n}$ is $1$-dimensional.
\end{corollary}
\begin{proof}
For any $n\in\Z_{\ge0}$, by Proposition \ref{prop:2020-08-18-1323}, there exist eigenvectors for $V_S|_{X_n}$ corresponding to the eigenvalues $1,\dots,2^{-n}$ in $X_n$. 
Since eigenvectors of distinct eigenvalues are linearly independent and the dimension of $X_n$ is $n+1$, we conclude that these eigenvectors constitute a basis of $X_n$, in particular, the eigenspace of $2^{-n}$ is $1$-dimensional.
\end{proof}
We denote by $\Phi_n\in X_n$ the unique eigenvector of $V_S|_X$ corresponding to $2^{-n}$ that is a monic polynomial of variable $h$ of degree $n$.
We define the dual basis $\{\Phi_n'\}_{n=0}^\infty \subset X'$ defined as continuous linear functionals on $X$ such that $\langle \Phi'_i | \Phi_j\rangle=\delta_{i,j}$.
Since $U_S^* = V_S$ is continuous on $X$ due to Prop.\ref{prop:2020-08-18-1323},
the dual operator $U_S^\times : X'\to X'$ of the Koopman operator is defined and continuous on $X'$ (Section 2).

\begin{theorem}
\label{thm: generalized spectrum for the one-sided full 2-shift}\,
\\
(1) The resolvent $(\lambda -V_S)^{-1}$ has an analytic continuation
to the inside of the unit circle as an operator from $X$ into $X'$.
The generalized eigenvalues with respect to the Gelfand triplet $X \subset L^2(\Sigma_+, \mu_+) \subset X'$ are $2^{-k},\,\, k=0,1,2,\cdots $ and the associated generalized eigenvectors are $\Phi_k, \,\,k=0,1,2,\cdots $.
\\
(2) $V_{S}$ and $U_S^\times$ have spectral decompositions of the form
\begin{eqnarray}
V_Sf &=& \sum^\infty_{i=0} 2^{-i} \overline{\langle \Phi_i' \,|\, f \rangle} \Phi_i, \quad f\in X, \label{thm3-6-1}\\
U_S^\times \mu &=& \sum^\infty_{i=0} 2^{-i} \langle \mu \,|\, \Phi_i \rangle \Phi_i', \quad \mu \in X'. \label{thm3-6-2}
\end{eqnarray}
\end{theorem}
\begin{proof}
(1) For $f,g\in L^2(\Sigma_+,\mu_+)$, we have
\begin{align*}
    \left((\lambda-V_S)^{-1}f ,g \right)
      =\frac{1}{\lambda}\left( (1-V_S/\lambda)^{-1}f , g\right).
\end{align*}
When $|\lambda|>1$, $\|V_S/\lambda\|<1$, and the Neumann series is applied to yield
\begin{align*}
    \left((\lambda-V_S)^{-1}f ,g \right)&=\frac{1}{\lambda} \sum_{r=0}^\infty \frac{1}{\lambda^{r}} \left( V_S^r f , g \right).
\end{align*}
Let us construct an analytic continuation as an operator from $X$ to $X'$.  For this purpose, suppose $f,g\in X$. Let $f=\sum_{i=0}^m c_i \Phi_i$ 
(note that the number $m$ depends on $f$).
Then, $V_S \Phi_i = \Phi_i/2^i$ shows
\begin{align*}
    \langle (\lambda-V_S)^{-1}f \, | \, g \rangle
    = \frac{1}{\lambda}\sum_{r=0}^\infty \sum_{i=0}^\infty \frac{c_i}{(2^{i}\lambda)^{r}}\langle \Phi_i | g \rangle
    = \sum_{i=0}^\infty \frac{c_i \langle \Phi_i | g\rangle}{\lambda-2^{-i}}.
\end{align*}
Since this equality holds for any $g\in X$, we obtain
\begin{eqnarray}
(\lambda-V_S)^{-1}f = \sum^\infty_{i=0} \frac{c_i}{\lambda-2^{-i}} \Phi_i.
\end{eqnarray}
This gives the analytic continuation of $(\lambda-V_S)^{-1}f$, called the generalized resolvent,  from the region $|\lambda|>1$ to $|\lambda|<1$ as an $X'$-valued function for each $f\in X$.
By the definition (Section 2), its pole $2^{-i}$ is a generalized eigenvalue and 
the image of the Riesz projection $c_i\Phi_i$ is a generalized eigenvector
for $i=0,1,2,\cdots$. 

(2) For the spectral decomposition of $V_S$, we again put $f=\sum_{i=0}^m c_i\Phi_i\in X$.
By applying $\Phi_n'$ in both sides, we immediately find $\overline{c_i} = \langle \Phi_i' \,|\, f \rangle$.
This and $V_S \Phi_i = \Phi_i/2^i$ proves (\ref{thm3-6-1}).
Further, (\ref{thm3-6-1}) yields
\begin{eqnarray*}
\langle \mu \,|\, V_Sf \rangle = \langle \mu \,|\, \sum^\infty_{i=0} 2^{-i} c_i \Phi_i \rangle 
= \sum^\infty_{i=0} 2^{-i} \langle \mu \,|\, \Phi_i \rangle \cdot \langle \Phi_i' \,|\, f \rangle.
\end{eqnarray*}
for any $\mu \in X'$.
On the other hand, we have 
\begin{eqnarray*}
\langle \mu \,|\, V_Sf \rangle = \langle \mu \,|\, U^*_Sf \rangle = \langle U_S^\times \mu \,|\, f \rangle.
\end{eqnarray*}
Therefore,
\begin{eqnarray*}
\langle U_S^\times \mu \,|\, f \rangle 
  = \Bigl\langle \sum^\infty_{i=0} 2^{-i} \langle \mu \,|\, \Phi_i \rangle \Phi_i' \,\Bigl|\, f \Bigr\rangle
\end{eqnarray*}
holds for any $f\in X$, which proves (\ref{thm3-6-2}).
\end{proof}

As a corollary of Theorem \ref{thm: generalized spectrum for the one-sided full 2-shift}, the iteration $V_S^k f$ converges to the average of $f$:
\begin{corollary}
The notation is as in Theorem \ref{thm: generalized spectrum for the one-sided full 2-shift}.
For any $f \in X$, we have 
\[V_S^kf \to \int_{\Sigma_+} f \dd\mu_+ \cdot \Phi_0\]
as $k \to \infty$ (recall $\Phi_0 \equiv 1$).
\end{corollary}
\begin{proof}
Let $f = \sum_{i=0}^m c_i \Phi_i$.
Then, we have
\[V_S^k f = \sum_{i=0}^m c_i 2^{-ik}\Phi_i.\]
Thus, $V_S^k f \to c_0$ as $k \to \infty$.
It suffices to show that 
\begin{align*}
    \int_{\Sigma_+} f \dd\mu_+ = c_0.
\end{align*}
It follows from the following formula: 
\begin{align}
    \int_{\Sigma_+} \Phi_n \dd\mu_+ = 0 \text{ for any $n \ge 1$.} \label{average of Phi}
\end{align}
In fact, since $\Phi_0 = 1$, we have
\begin{align*}
    \int_{\Sigma_+} \Phi_n \dd\mu_+ 
    &= \int_{\Sigma_+} \Phi_n \Phi_0\dd\mu_+ \\
    &= \int_{\Sigma_+} \Phi_n U_S[\Phi_0]\dd\mu_+ \\
    &= \int_{\Sigma_+} V_S[\Phi_n] \Phi_0\dd\mu_+ \\
    &= 2^{-n}\int_{\Sigma_+} \Phi_n \Phi_0\dd\mu_+ \\
    &= 2^{-n}\int_{\Sigma_+} \Phi_n \dd \mu_+.
\end{align*}
Thus, we have \eqref{average of Phi}.
\end{proof}

\begin{remark}
Antoniou and Tasaki \cite{AT92} computed the generalized eigenvalues for the R\'enyi map.
They prove that the generalized eigenvalues of Perron-Frobenius operator associated to 2-adic R\'enyi map are $1,2^{-1},2^{-2},\dots$, and the corresponding eigenvectors are {\em Bernoulli polynomials} $\{B_n(x)\}_{n=0}^\infty$, which are defined as coefficients of the following Taylor expansion:
\[\frac{te^{xt}}{e^t-1}=\sum_{n=0}^\infty \frac{B_n(x)}{n!}t^n.\]
Theorem \ref{thm: generalized spectrum for the one-sided full 2-shift} is compatible with their result through the isomorphism $h^*$ of (\ref{eq:isom of L^2}).  Namely, the Bernoulli polynomial $h^*B_n$ is identical with $\Phi_n$.
In fact, the Bernoulli polynomial satisfies (a special case of) Raabe's multiplication theorem\cite{Raa1851}:
\[\frac{1}{2}B_n(x) + \frac{1}{2}B_n(x+1/2) = 2^{-n}B(2x).\]
We can prove that the left hand side coincides with $(V_TB_n)(2x)$, and thus, Raabe's multiplication theorem means that the Bernoulli polynomial $B_n$ is an eigenvector of $V_T$ corresponding to $2^{-n}$. 
Since the test space $X$ in $L^2(\Sigma_+,\mu_+)$ corresponds to the space of polynomials in $L^2([0,1],\dd{x})$ via $h^*$, 
by the same argument of the proof of Proposition \ref{prop:2020-08-18-1323}, $B_n$ is characterized as a unique eigenvector with respect to $2^{-n}$ in the space of polynomials, and thus $h^*B_n=\Phi_n$.
\end{remark}
\subsection{The one-sided full $\beta$-shift}
\label{subsec: full beta shift}
In this subsection, we consider $\Sigma_+ := [\beta]^{\mathbb{N}}$ and the one-sided full $\beta$-shift $S$ on $\Sigma_+$. The generalized eigenvalues of the Perron-Frobenius operator $V_S$ on $L^2(\Sigma_+, \mu_+)$ with a general Bernoulli measure $\mu_+$ will be obtained. The outline of the proof is carried out in the same way as the case of the one-sided full $2$-shift described in the previous subsection, thus we here only provide corresponding statements without proofs.

First, we specify some notations. 
Let $\beta \ge 2$ be a positive integer.
Let $p_0,\dots, p_{\beta-1}$ be positive numbers such that $p_0+\dots+p_{\beta-1}=1$, and let $\mu_0$ be a measure on $[\beta]=\{0,1,\dots,\beta-1\}$ defined by $\mu_0(\{i\})=p_i$.
Then, the product measure on $\Sigma_+$ is denoted by $\mu_+ := \mu_0^{\otimes \N}$.  

Fix an orthonormal system $\psi_0=1, \psi_1,\dots,\psi_{\beta-1}$ of $L^2([\beta],\mu_0)$.
For example, we may take $\psi_s(\omega) = {\rm e}^{2\pi i s\omega / \beta}$ if $p_0 = \cdots = p_{\beta-1} = 1/\beta$.
For $n \in \Z_{\ge 0}$, we define
\begin{align}
    W_n(\omega_1, \omega_2, \dots) := \prod_{k=1}^\infty \psi_{s_k}(\omega_k),
\end{align}
where $n=\sum_{k=1}^\infty s_k\beta^{k-1}$ is the $\beta$-adic expansion of $n$. 
The functions $W_n$'s are generalizations of Walsh functions.

Then, we have the following proposition:
\begin{proposition}
\begin{enumerate}
    \item \label{PF op of one full shift 2}For $f\in L^2(\Sigma_+,\mu_+)$, we have 
    \begin{align}
        (V_{S}f)(\omega)=\sum_{i=0}^{\beta-1}p_if(i*\omega).
    \end{align}
    \item \label{PF op and Walsh function 2}For $n\in\Z_{\ge0}$, we have
    \begin{align}
    U_{S}W_n&=W_{\beta n},\\
    V_{S}W_n&=
    \begin{cases}
        W_{n/\beta} & \text{ if }n\in \beta\Z,\\
        0 & \text{otherwise}.
    \end{cases}
    \end{align}
    \item \label{spc of PF op 2}The spectrum of $V_{S}$ is described as follows:
    \begin{align*}
        \text{point spectrum: }& \{z \in \C \,|\, |z|<1\}\cup \{1\},\\
        \text{continuous spectrum: }& \{z \in \C \,|\, |z|=1, z\neq 1\}.
    \end{align*}
    The eigenspace of the eigenvalue $z = 1$ is one dimensional space consisting of constant functions. Otherwise, the eigenspaces of the point spectrums $z$ are infinite dimensional.
\end{enumerate}
\end{proposition}

For $n \ge 1$, and $\omega=(\omega_1,\omega_2,\dots)\in\Sigma_+$, we define 
\[N_n^k(\omega)=\#\!\left \{ i \in \{1,2,\dots,n-1\} \,|\, \omega_i=k\right\}, \]
and $N_1^k(\omega)=0$ for all $k$. 
Then, we define $h:\Sigma_+\rightarrow [0,1]$, which has the same properties of the map (\ref{eq: def of h}), by
\begin{align}
    h(\omega) = \sum_{\substack{n\ge1 \\ \omega_n\neq0}} 
    (p_0+\dots+p_{\omega_n-1})p_0^{N_{n}^0(\omega)} \cdots p_{\beta-1}^{N_{n}^{\beta-1}(\omega)}.
    \label{eq: def of h general}
\end{align}
\begin{remark}
Regarding the definition of the map \eqref{eq: def of h general}, let us consider the following procedure to shrink intervals: for any closed interval $I=[a,b]$ and $\omega \in [\beta]$, define
\[I_\omega := \left[a+(p_0+\dots+p_{\omega-1})|I|, a + (p_0+\dots+p_{\omega})|I|\right] \subset I,\]
where $|I| := b-a$.
We here decompose $I$ into $\beta$ small intervals $I_\omega$ ($\omega \in [\beta]$) and $I_\omega$ is the $\omega$-th interval of length $p_{\omega}|I|$.
For $\omega_1,\dots, \omega_n \in [\beta]$, let $I_{\omega_1,\dots,\omega_n} := (\dots(I_{\omega_1})_{\omega_2} \dots)_{\omega_n}$.
Then, $h(\omega)$ is characterized as a unique point of the intersection of intervals:
\[\{h(\omega)\} = \bigcap_{n=1}^\infty I_{\omega_1,\dots, \omega_n},\]
where $\omega = (\omega_1,\omega_2, \dots) \in \Sigma_+$.
\end{remark}
Then, we have a similar proposition to Proposition \ref{prop: behavior of h via PF op} as follows:
\begin{proposition}
For any $i\in[\beta]$, 
\begin{align*}
    h(i*\omega) &= p_0+\cdots+p_{i-1} + p_ih(\omega).
\end{align*}
In particular, we have
\begin{align*}
    V_{S}[h] &= \bigg(\sum_{i\in[\beta]}p_i^2\bigg)h+\sum_{0\le i<j\le\beta-1}p_ip_j, \\
    V_{S}[h^n] &= \bigg( \sum_{i\in[\beta]} p_i^{n+1}\bigg)h^n + q(h).
\end{align*}
where $q(h)\in \C[h]$ is a polynomial of variable $h$ of degree smaller than $n$.
\end{proposition}
We define a test space $X$, eigenvectors $\Phi_n\in X$, their dual elements $\Phi'_n\in X'$, and the dual operator $U_S^\times : X' \to X'$ of the Koopman operator in the same way illustrated in the previous subsection.
Then, in a similar manner to Theorem \ref{thm: generalized spectrum for the one-sided full 2-shift}, we have the following result:
\begin{theorem}\label{thm: generalized spectrum for the one-sided full beta-shift}\,
\\
(1) The resolvent $(\lambda -V_S)^{-1}$ has an analytic continuation
to the inside of the unit circle as an operator from $X$ into $X'$.
The generalized eigenvalues with respect to the Gelfand triplet $X \subset L^2(\Sigma_+, \mu_+) \subset X'$ are given by $\{\sum_{i\in[\beta]}p_i^{n+1}\}_{n\ge0}$.
\\
(2) $V_{S}$ and $U_S^\times$ have spectral decompositions of the form
\begin{eqnarray}
V_Sf &=& \sum^\infty_{i=0}\bigg(\sum_{n\in[\beta]}p_n^{i+1}\bigg) \overline{\langle \Phi_i' \,|\, f \rangle} \Phi_i, \quad f\in X, \label{thm3-12-1}\\
U_S^\times \mu &=& \sum^\infty_{i=0} \bigg(\sum_{n\in[\beta]}p_n^{i+1}\bigg) \langle \mu \,|\, \Phi_i \rangle \Phi_i', \quad \mu \in X'. \label{thm3-12-2}
\end{eqnarray}
\end{theorem}

As a corollary, we have an asymptotic behavior of the Perron-Frobenius operator:
\begin{corollary}
For any $f \in X$, we have 
\[V_S^kf \to \int_{\Sigma_+} f \dd\mu_+ \cdot \Phi_0\]
as $k \to \infty$.
\end{corollary}

\subsection{An one-sided partial subshift}
\label{subsec: partial shift}
In this subsection, we consider the shift map $S$ on $\Sigma_+^A$ associated to a one-sided subshift defined by the adjacency matrix 
\[ A=
\left(
\begin{array}{cc}
    1 & 1 \\
    1 & 0
\end{array}
\right),
\]
which is introduced in the beginning of Section \ref{sec: generalized spectrum of one-sided shifts}.
For $i_1,\dots,i_r = 0, 1$, we define a cylinder set $C[i_1, \dots, i_r]\subset \Sigma_+^A$ by
\[ C[i_1,\dots, i_r] := \left\{(\omega_1,\omega_2, \dots, )\in \Sigma_+^A \,|\, \omega_k = i_k \text{ for }k=1,\dots,r\right\}.\]
We regard $\Sigma_+^A$ as a measurable space with $\sigma$-algebra generated by the above cylinder sets.

At first, we fix an invariant measure on $\Sigma_+^A$.
Let $\varphi:=(1+\sqrt{5})/2$ be the golden ratio.
Then, we define a transition matrix $P=(p_{ij})_{i,j=0,1}$, and a probability vector $\pi=(\pi_0, \pi_1)$ as follows:
\[ P=
\left(
\begin{array}{cc}
     \varphi^{-1}& \varphi^{-2}  \\
     1& 0
\end{array}
\right),~\pi=\frac{1}{\varphi^2+1}(\varphi^2, 1).\]
We define the {\em Markov measure} $\mu_+$ associated to $P$ by
\[\mu_+(C[i_1,\dots,i_r]) := \pi_{i_1} p_{i_1\, i_2}\cdots p_{i_{r-1}\, i_r} \]
for arbitrary finite $i_1,\dots, i_r = 0,1$.
Since $(P, \pi)$ is stationary (i.e. $\pi P = \pi$), $\mu_+$ is an invariant measure on $(\Sigma_+^A, S)$~\cite{CFS, Wal}.
We define a map $h: \Sigma_+^A \rightarrow [0,1]$ by
\[h(\omega)= \sum_{k=1}^\infty \frac{\omega_k}{\varphi^{k}}.\]
We remark that $h$ is continuous and surjective, and induces a homeomorphism into its image outside a countable subset of $\Sigma_+^A$.

\begin{proposition}
Let $\rho \in L^1([0,1])$ be a density function corresponding to the push-forward measure $h_*\mu_+$, namely, $h_*\mu_+ = \rho(x)\dd{x}$.
Then, we have
\begin{align}\label{3-24}
    \rho(x) =
    \begin{cases}
        \displaystyle \frac{\varphi^3}{1+\varphi^2} & \text{ if } 0\le x \le \frac{1}{\varphi},\\[13pt]
        \displaystyle \frac{\varphi^2}{1+\varphi^2} & \text{ if }\frac{1}{\varphi} < x \le 1.
    \end{cases}
\end{align}
\end{proposition}
\begin{proof}
We first show the following formula:
\begin{align}
h^{-1}([a,b]) = C[i_1,\dots,i_k]. \label{h inverse a b equal C i}
\end{align}
Here,
\begin{align}
\label{ab no katachi}
\begin{aligned}
    a &=h(i_1,\dots,i_k,0,0,\dots), \\
    b &=h(i_1, \dots, i_k,s,t,s,t,\dots),
\end{aligned}
\end{align}
where $(s,t)=(1,0)$ if $i_k=0$ and $(s,t)=(0,1)$ if $i_k=1$.
Here, note that a sequence $(\cdots ,1,1, \cdots)$ is impossible due to the definition of $A$.
The inclusion $\supset$ is obvious. 
Conversely, we easily see that if $\omega \notin C[i_1,\dots, i_k]$, then $h(\omega)<a$ or $h(\omega)>b$.
Thus, we have the formula \eqref{h inverse a b equal C i}. 
By a direct calculation, we have
\[\mu_+(C[i_1,\dots, i_k]) = 
\begin{cases}
(1 + \varphi^2)^{-1} \varphi^{3-k} & \text{ if } (i_1, i_k) = (0,0),\\
(1 + \varphi^2)^{-1} \varphi^{2-k} & \text{ if }(i_1, i_k) = (0,1),\\
(1 + \varphi^2)^{-1} \varphi^{2-k} & \text{ if } (i_1, i_k) = (1,0),\\
(1 + \varphi^2)^{-1} \varphi^{1-k} & \text{ if }(i_1, i_k) = (1,1).\\
\end{cases}
\]
We can verify that $\mu_+(C[i_1,\dots, i_k])$ above and $\rho(x)$ given by (\ref{3-24}) satisfies
\[\int_{a}^b \rho(x) \dd{x} = \mu_+(C[i_1,\dots, i_k]).\]
Since $[a,b]$'s for any $a,b$ in the form of \eqref{ab no katachi} generate the Borel $\sigma$-algebra of $[0,1]$, we obtain $h_*\mu_+ = \rho(x) \dd{x}$.
\end{proof}
\begin{remark}
We note that the shift map $S$ on $\Sigma_+^A$ is conjugate to the multiplication mapping $T_\varphi$ by the golden ratio on $[0,1]$: $x\mapsto \varphi x \mod 1$ in the sense of $U_Sh^*=h^*U_{T_\varphi}$.
Then, $\rho(x) \dd{x}$ is an invariant measure for $T_\varphi$.
\end{remark}
For the Perron-Frobenius operator $V_S$, we have the following formula:
\begin{proposition}
\label{prop: VS for partial shift}
For $f \in L^2(\Sigma_+^A, \mu_+)$, we have
\begin{align}
    V_S[f](\omega) 
    = \frac{1}{\varphi}\mathbf{1}_{C[0]}(\omega) f(0*\omega) 
    + \frac{1}{\varphi^2}\mathbf{1}_{C[0]}(\omega) f(1*\omega) 
    + \mathbf{1}_{C[1]}(\omega) f(0*\omega).
\end{align}
\end{proposition}
\begin{proof}
Let $f, g\in L^2(\Sigma_+^A, \mu_+)$ be arbitrary elements.
Then, we have
\begin{align*}
    \int_{\Sigma_+^A} f \overline{U_Sg} \dd\mu_+ = \sum_{i=0,1} \int_{C[i]}f \overline{U_Sg} \dd\mu_+. 
\end{align*}
For $i=0,1$, we denote by $\sigma_i$ the map $S|_{C[i]}: C[i] \rightarrow  \Sigma_+^A$.
Then, we have
\begin{align*}
   \sum_{i=0,1} \int_{C[i]}f \overline{U_Sg} \dd\mu_+ &= \sum_{i=0,1} \int_{\Sigma_+^A}f(i*\omega) \overline{g(\omega)} \dd(\sigma_i)_*\mu_+(\omega) \\
    &= \sum_{i,j=0,1} \int_{C[j]} f(i*\omega) \overline{g(\omega)} \dd(\sigma_i)_*\mu_+(\omega).
\end{align*}
For any $i_1,\dots, i_k = 0,1$, we have 
\begin{align*}
    (\sigma_i)_*(\mu_+|_{C[i]})(C[j,i_1,\dots,i_k]) &= \mu_+(C[i,j,i_1,\dots,i_k]) \\
    &=\pi_i p_{ij}p_{ji_1}\cdots p_{i_{k-1}i_k} \\
    &=\frac{\pi_i p_{ij}}{\pi_j} \mu_+(C[j,i_1,\dots,i_k]).
\end{align*}  
Thus, we see that
\begin{align*}
    \sum_{i,j=0,1} \int_{C[j]} f(i*\omega) \overline{g(\omega)}\,\dd(\sigma_i)_*\mu_+(\omega) 
    = \sum_{i,j = 0,1}\int_{C[j]}f(i*\omega) \overline{g(\omega)} \frac{\pi_i p_{ij}}{\pi_j} \,\dd\mu_+(\omega).
\end{align*}
Therefore, we have
\begin{align*}
    &\int_{\Sigma_+^A} V_S[f] \overline{g} \,\dd\mu_+\\
    &=\int_{\Sigma_+^A} f \overline{U_S[g]} \,\dd\mu_+ \\
    &= \int_{\Sigma_+^A} \left( \sum_{i,j=0,1}\frac{\pi_i p_{ij}}{\pi_j} \mathbf{1}_{C[j]}(\omega)f(i*\omega)\right) \overline{g(\omega)} \,\dd\mu_+(\omega)\\
    & = \int_{\Sigma_+^A} \left\{ \frac{\mathbf{1}_{C[0]}(\omega)f(0*\omega)}{\varphi} 
    + \frac{\mathbf{1}_{C[0]}(\omega) f(1*\omega)}{\varphi^2} 
    + \mathbf{1}_{C[1]}(\omega)f(0*\omega)\right\} \\
    &\hspace{50pt} \times \overline{g(\omega)} \dd\mu_+(\omega).
\end{align*}

\end{proof}

\paragraph{Calculation of the generalized eigenvalues.}
We compute the generalized eigenvalues for the subshift.
We specify the test space $X$.
Let \[W := \mathbb{C} \mathbf{1}_{C[0]} \oplus_\mathbb{C} \mathbb{C} \mathbf{1}_{C[1]}\] be a 2-dimensional subspace of $L^2(\Sigma_+^A, \mu_+)$.
Then, we define 
\begin{align*}
    X_n &:= W \otimes \bigoplus_{i=0}^n \mathbb{C}h^i, \\
    X &:= \lim_{\underset{n}{\longrightarrow}} X_n = W\otimes_\mathbb{C} \mathbb{C}[h].
\end{align*}

We have a similar proposition to Proposition \ref{prop:2020-08-18-1323}:
\begin{proposition}
For $n\in\Z_{\ge0}$, we have
\[V_{S}X_n\subset X_n.\]
Moreover, the representation matrix of $V_{S}|_{X_n}$ associated to the basis 
\[\mathbf{1}_{C[0]},\mathbf{1}_{C[1]},\dots,\mathbf{1}_{C[0]}h^n, \mathbf{1}_{C[1]}h^n\]
 is in the form of
\[
\left(
\begin{array}{cccc}
    Q_0 &  &* & \\
    0 & Q_1 &  &* \\
    \vdots &   \ddots & \ddots & \\
    0 & \cdots & 0 & Q_n 
\end{array}
\right),
\]
where 
\[
Q_n:=\left(
\begin{array}{cc}
    \varphi^{-n-1} & \varphi^{-n-2} \\
    \varphi^{-n} & 0
\end{array}
\right).
\]
\end{proposition}
\begin{proof}
It follows from a direct calculation with Proposition \ref{prop: VS for partial shift}.
\end{proof}
Then, we have the following corollary using the same argument of the proof of Corollary \ref{cor: 202008241553}:
\begin{corollary}
\label{cor: gen spec for par shift}
The eigenvalues of $V_S|_X$ on $X$ are $\varphi^{-n}, -\varphi^{-n-2}$ for $n\ge0$, which are eigenvalues of $Q_n$'s.
Moreover, each eigenspace is one-dimensional.
\end{corollary}

We fix $\Phi_{n,+}$ and $\Phi_{n,-}$ the unique (up to scalar) eigenvectors of $V_S|_X$ corresponding to $\varphi^{-n}$ and $-\varphi^{-n-2}$, respectively.
We assume $\Phi_{0,+}=1$.
We define the dual basis $\{\Phi_{n,\pm}'\}_{n=0}^\infty$, 
which are continuous linear functionals on $X$ such that 
$\langle \Phi_{i,p}' | \Phi_{j,q}\rangle = \delta_{(i,p),(j,q)}$, where $i,j\ge 0$ and $p,q=\pm$.

Then, we have the analytic continuation of the resolvent and the following spectral decompositions of the Perron-Frobenius and Koopman operators via the generalized eigenvalues:
\begin{theorem}
\label{thm: generalized eigenvalues for the one-sided partial 2-shift}\,
\\
(1) The resolvent $(\lambda -V_S)^{-1}$ has an analytic continuation
to the inside of the unit circle as an operator from $X$ into $X'$.
The generalized eigenvalues with respect to the Gelfand triplet $X \subset L^2(\Sigma^A_+, \mu_+) \subset X'$ are given by $\{\varphi^{-n}\}_{n\ge0} \cup \{-\varphi^{-n-2}\}_{n\ge0}$.
\\
(2) $V_{S}$ and $U_S^\times$ have spectral decompositions of the form
\begin{eqnarray*}
V_Sf &=& \sum^\infty_{n=0} \varphi^{-n} \overline{\langle \Phi_{n,+}' \,|\, f \rangle} \Phi_{n,+} - \sum^\infty_{n=0} \varphi^{-(n+2)} \overline{\langle \Phi_{n,-}' \,|\, f \rangle} \Phi_{n,-}, \quad f\in X, \\
U_S^\times \mu &=& \sum^\infty_{n=0} \varphi^{-n} \langle \mu \,|\, \Phi_{n,+} \rangle \Phi_{n,+}' - \sum^\infty_{n=0} \varphi^{-(n+2)} \langle \mu \,|\, \Phi_{n,-} \rangle \Phi_{n,-}', \quad \mu \in X'.
\end{eqnarray*}
\end{theorem}
The proof of this theorem is carried out via the same argument as that of Theorem \ref{thm: generalized spectrum for the one-sided full 2-shift}.
As a corollary, we have an asymptotic behavior of the Perron-Frobenius operator:
\begin{corollary}
For any $f \in X$, we have 
\[V_S^kf \to \int_{\Sigma_+} f \dd\mu_+ \cdot \Phi_{0,+}\]
as $k \to \infty$ ($\Phi_{0,+} \equiv 1$).
\end{corollary}


\section{Generalized eigenvalues of a two-sided shift}
In this section, we determine the generalized eigenvalues of the left shift operator $(S_L\omega )_j = \omega _{j+1}$
on the two-sided full-shift space
\begin{equation}
\Sigma = \{ \omega = (\cdots , \omega _{-1}, \omega _0.\, \omega _1, \cdots )\, | \, 
   \omega _i\in \{ 0,1 \} \}.
\label{4-1}
\end{equation}
In this case, $S_L$ is a homeomorphism and its inverse is the right shift $(S_R\omega )_j := \omega _{j-1}$.
The precise statement is included  in Theorem \ref{main thm} under the preparation in the next subsection.


\subsection{Preliminary}
For $\omega = (\omega_i)_{i\in \mathbb{Z}} \in \Sigma$, we denote by $\omega_+$ (resp. $\omega_-$) a half side $(\omega_1,\omega_2,\dots)$ (resp. $(\omega_0, \omega_{-1},\dots)$) of $\omega$.
Note that in (\ref{4-1}), we use dot $(.)$ instead of $(,)$ before $\omega_1$ to distinguish the left and right half sides.
We define
\begin{align}
    \Sigma_{\pm} & := \left\{\omega_{\pm} | \omega \in \Sigma \right\} \cong \{0,1\}^{\mathbb{N}},\\
    \pi_{\pm}&: \Sigma \longrightarrow \Sigma_{\pm}; \omega \mapsto \omega_{\pm}.
\end{align}
Let $\mu_0$ be the measure as in Section \ref{subsec: full 2 shift} : $\mu_0(\{0\})=\mu_0(\{1\})=1/2$.
We define measures $\mu$ and $\mu_{\pm}$ on $\Sigma$ and $\Sigma_{\pm}$ as the product measures of $\mu_0$, respectively.
For simplicity, $L^2 (\Sigma, \mu)$ and $L^2 (\Sigma_{\pm}, \mu_{\pm})$ are 
denoted by $L^2 (\Sigma)$ and $L^2 (\Sigma_{\pm})$, respectively.
We note the following isomorphism:
\begin{align}
     L^2(\Sigma_+)\otimes L^2(\Sigma_-) \cong L^2(\Sigma); f\otimes g \mapsto [\omega\mapsto f(\omega_+)g(\omega_-)],
\end{align}
where the Hilbert tensor product on the left hand side is the completion of the algebraic tensor product of $L^2(\Sigma_+)$ and  $L^2(\Sigma_-)$ with respect to the inner product induced via the correspondence $( v\otimes w, v'\otimes w') := ( v,v') \cdot (w,w')$ for $v,v' \in L^2(\Sigma_+)$ and $w,w' \in L^2(\Sigma_-)$.
For simplicity, we denote the Koopman operators of the left and right shifts by $U_L$ and $U_R$ 
\begin{eqnarray*}
& & (U_L f)(\cdots , \omega _{-1}, \omega _0. \omega _1,\cdots ) = f(\cdots ,\omega _0, \omega _1.\, \omega _2, \cdots ) \\
& & (U_R f)(\cdots , \omega _{-1}, \omega _0. \omega _1,\cdots ) = f(\cdots ,\omega _{-2}, \omega _{-1}.\, \omega _0, \cdots ),
\end{eqnarray*}
respectively.
The Perron-Frobenius operators of them are the adjoint denoted by $V_L = U_L^*$ and $V_R= U_R^*$.
These operators are unitary, and satisfy $V_L=U_R$ and $V_R=U_L$.
Therefore, the action of $V_L$ is the right shift
\begin{equation}
(V_Lf)(\omega ) = f(\cdots ,\omega _{-2}, \omega _{-1}.\, \omega _0,\cdots ).
\label{PF_two_side}
\end{equation}
We also define $U_{L+}$ (resp. $U_{R-}$) as the Koopman operator on the left (resp. right) shift map on $\Sigma_+$ (resp. $\Sigma_-$) and its adjoint is denote by $V_{L+}$ (resp. $V_{R-}$).
It is easy to verify that
\begin{eqnarray*}
U_{L} \circ \pi_+^* = \pi^*_+ \circ U_{L+}, \quad U_{R} \circ \pi_-^* = \pi^*_- \circ U_{R-}.
\end{eqnarray*}
Let $h: \Sigma_+\rightarrow [0,1] \subset \mathbb{C}$ be the continuous function defined in (\ref{eq: def of h}).
Let 
\begin{align}
    P_M(\Sigma_+) &:= \bigoplus_{i=0}^M \mathbb{C} h^i,\\
    P(\Sigma_+) &:= \lim_{\underset{M}{\longrightarrow}} P_M(\Sigma_+) 
\end{align}
be topological vector spaces, which coincide with $X_M$ and $X$ of (\ref{eq: def of X_n}) and (\ref{eq: def of X}), respectively.
As mentioned, $P(\Sigma_+)$ is a Montel space, in particular complete and barrelled. 
There is a natural homeomorphism $i: \Sigma _- \to \Sigma _+$: 
\begin{eqnarray*}
\omega _{-} = (\cdots , \omega _{-1}, \omega _{0}) \mapsto \omega _+ = (\omega _0, \omega _{-1}, \cdots ).
\end{eqnarray*}
We define the space $P(\Sigma_-) := i^*P(\Sigma_+ )$ by the pullback.

Let $\{ \Phi_m\}_{m=0}^\infty \subset P(\Sigma_+)$ be the generalized eigenvectors of the one-sided subshift associated with the generalized eigenvalues $\{ 2^{-m}\}^\infty_{m=0}$ defined in the following paragraph of Corollary \ref{cor: 202008241553} and 
$\{\Phi'_m\}_{m=0}^\infty \subset P(\Sigma_+)'$ the dual basis of $\Phi_m$'s
satisfying $\langle \Phi'_m | \Phi_n\rangle=\delta_{m,n}$.
Further, we define $\Psi_m = i^* \Phi_m \in P (\Sigma_-) $ and its dual element $\Psi_m' \in P (\Sigma_-)'$ satisfying $\langle \Psi'_m | \Psi_n\rangle=\delta_{m,n}$.

\begin{lemma}
\label{lem: eigen function formula}
The following identities hold:
\begin{align*}
    V_{L+}\Phi_m &= 2^{-m}\Phi_m,\\
    V_{R-}\Psi_m &= 2^{-m}\Psi_m,\\
    U^{\times}_{L+}\Phi'_m &= 2^{-m}\Phi'_m,\\
    U^{\times}_{R-}\Psi'_m &= 2^{-m}\Psi'_m.
\end{align*}
\end{lemma}
\begin{proof}
The first and second equalities follow by the definition (Section 3.1).
We prove the third equality for the dual operator $U^{\times}_{L+}\Phi'_m = 2^{-m}\Phi'_m$.
The last formula is proved in the same way.
In fact, for arbitrary $f \in P(\Sigma_+)$, put $f=\sum_{i=0}^n c_i\Phi_i$.
Noting the definition of the dual operator given in Section 2, we have
\begin{align*}
    \langle U^{\times}_{L+}\Phi'_m~|~f\rangle 
    &= \langle \Phi'_m~|~ U^*_{L+}f\rangle \\
    &= \langle \Phi'_m ~|~ V_{L+}\sum_{i=0}^n c_i \Phi_i\rangle \\
    &= \langle \Phi'_m ~|~ \sum_{i=0}^n c_i2^{-i}\Phi_i\rangle \\
    &=2^{-m}\langle \Phi'_m~|~f\rangle,
\end{align*}
which gives the desired formula.
\end{proof}

As in Section 3, we have the Gelfand triplets
\begin{eqnarray}
P(\Sigma_+ ) \subset L^2(\Sigma_+ ) \subset P(\Sigma_+ )', \quad 
P(\Sigma_- ) \subset L^2(\Sigma_- ) \subset P(\Sigma_- )'.
\label{4-7}
\end{eqnarray}

\begin{lemma}\label{lemma4-2}
Any $f\in L^2(\Sigma_+ )$ is expanded as
\begin{equation}
f = \sum^\infty_{m=0} c_m \Phi_m', \quad c_m = \langle f~|~\Phi_m\rangle,
\end{equation}
with respect to the weak * topology on $P(\Sigma_+ )'$.
We also obtain a similar statement: any element of $L^2(\Sigma_-)$ is expressed as an infinite sum of elements in $P(\Sigma_-)'$ with respect to the weak * topology.
\end{lemma}
\begin{proof}
It is obvious that $c_m$ is given as above if the right hand side exists.
For the convergence of the right hand side in $P(\Sigma_+)'$, it is sufficient to show that $\langle f \,|\, g  \rangle$ exists for any $g\in P(\Sigma _+)$.
By the definition of $P(\Sigma _+)$, there is a natural number $M$ such that
$g = \sum^M_{n=0}d_n \Phi_n$.
Hence, 
\begin{eqnarray*}
\langle f \,|\, g  \rangle = \sum^\infty_{m=0}\sum^M_{n=0} c_m \overline{d_n} \langle \Phi'_m \,|\, \Phi_n  \rangle
 = \sum^M_{n=0}c_n\overline{d_n}< \infty.
\end{eqnarray*}
\end{proof}

Let $P(\Sigma_+)\otimes P(\Sigma_-)'$ and $P(\Sigma_+)' \otimes P(\Sigma_-)$ be 
algebraic tensor products.
We define the bilinear form on them by
\begin{align}
    \langle f_+ \otimes \mu_- | \mu_+\otimes f_-\rangle 
    :=\langle \mu_- | f_-\rangle \cdot \overline{\langle \mu_+ | f_+\rangle}, \label{the bilinear form}
\end{align}
for $ f_+ \otimes \mu_- \in P(\Sigma_+)\otimes P(\Sigma_-)'$
and $\mu_+\otimes f_- \in P(\Sigma_+)' \otimes P(\Sigma_-)$.
Then, we have the following proposition:
\noindent \begin{proposition}\, 

\label{prop: 04261347}
\begin{enumerate}
    \item The bilinear form (\ref{the bilinear form}) is compatible with the inner product on $L^2(\Sigma_+)\otimes L^2(\Sigma_-)$, namely, for $f_{\pm} \in P(\Sigma_{\pm})$ and $\mu_{\pm} \in L^2(\Sigma_{\pm}) \subset P(\Sigma_{\pm})'$, we have
    \[\langle f_+\otimes \mu_- | \mu_+ \otimes f_-\rangle = ( f_+\otimes \mu_- , \mu_+ \otimes f_-)_{L^2(\Sigma)}.\]
    \label{compatibility of bilinear form}
    \item $\langle \Phi_m \otimes \Psi_n' | \Phi'_{m'} \otimes \Psi_{n'} \rangle = \delta_{m,m'}\delta_{n,n'}$.
    \label{orthogonality}
\end{enumerate}
\end{proposition}
\begin{proof}
As for (\ref{compatibility of bilinear form}), 
since $\langle \mu \,|\, f \rangle = (\mu, f)_{L^2}$ when $\mu \in L^2$, we have
\begin{align*}
    (\text{the left hand side}) &= \int \mu_- \overline{f_-} \dd\mu_- \cdot \int \overline{\mu_+} f_+ \dd\mu_+ \\
    &= \int f_+(\omega_+)\mu_-(\omega_-)\overline{f_-(\omega_-)\mu_+(\omega_+)} \dd\mu_+ \dd\mu_-\\
    &= (f_+\mu_-, f_-\mu_+)_{L^2(\Sigma)}\\
    &= (\text{the right hand side}).
\end{align*}
The second statement follows from the definition.
\end{proof}

\begin{proposition}
Let $E$ be a Fr\'echet space and $F_M := E\otimes P_M(\Sigma_-)$.
Then, the strict inductive limit 
\[F:=\lim_{\underset{M}{\longrightarrow}} F_M = E\otimes P(\Sigma_-) \] 
is a LF-space \cite{Tre}, which is a complete barrelled space.
\end{proposition}

In this paper, we consider $E = L^2(\Sigma_{\pm})$ as follows.
\begin{definition}
Define
\[ X_+ := L^2(\Sigma_+)\otimes P(\Sigma_-), \quad X_- := P(\Sigma_+) \otimes L^2(\Sigma_-). \]
Then, we have the Gelfand triplets:
\begin{align*}
    X_+ \subset L^2(\Sigma) \subset X_+',\\
    X_- \subset L^2(\Sigma) \subset X_-'.
\end{align*}
Indeed, $L^2 (\Sigma) \simeq L^2(\Sigma_+) \otimes L^2(\Sigma_-)$ and
the triplet (\ref{4-7}) show that $X_{\pm}$ is dense in $L^2 (\Sigma)$
and its topology is stronger than that of $L^2 (\Sigma)$.
For $f= f_+\otimes f_- \in L^2(\Sigma_+)\otimes P(\Sigma_-) = X_+$, we have
$i^*f_+ \in L^2(\Sigma _-)$ and $(i^*)^{-1} f_- \in P(\Sigma_+)$,
so that $(i^*)^{-1} f_- \otimes i^*f_+ \in X_-$.
Therefore, $X_+ \simeq X_-$ and the following Gelfand triplets are also considered, which will play an important role for the analytic continuation of the resolvent.
\begin{align}
    X_+ \subset L^2(\Sigma) \subset X_-', \label{triplet +-}\\
    X_- \subset L^2(\Sigma) \subset X_+'. \label{triplet -+}
\end{align}
\end{definition}
\begin{lemma}
\label{lemma4-6}
The following inclusions hold:
\begin{equation}
P(\Sigma_+)\otimes P(\Sigma_-)' \subset X'_+, \quad P(\Sigma_+)' \otimes P(\Sigma_-) \subset X_-'.
\label{an inclusion 1}
\end{equation}
\end{lemma}
\begin{proof}
We show the first one.
Let $f_+ \oplus \mu_- \in P(\Sigma_+) \oplus P(\Sigma_-)'$
and $\mu_+ \oplus f_- \in L^2(\Sigma_+) \oplus P(\Sigma_-) = X_+$.
We have
\[\langle f_+ \otimes \mu_-\,|\,\mu_+ \otimes f_- \rangle = \langle \mu_-\,|\,f_-\rangle \cdot \overline{(\mu_+, f_+)}_{L^2(\Sigma_+)}.\]
By the definition of the inductive limit topology of $X_+ = L^2(\Sigma_+) \otimes P(\Sigma_-)$, it is easy to verify that the right hand side tends to zero
if $\mu_+ \to 0$ in $L^2 (\Sigma_+)$ or $f_- \to 0$ in $P(\Sigma_-)$.
This implies that the left hand side is continuous with respect to 
$\mu_+ \otimes f_- \in X_+$, which proves that $ f_+ \otimes \mu_- \in X_+'$.
\end{proof}

\begin{proposition}\label{prop4-7}
For $f\in X_-$ and $g \in X_+$, we have expansions:
\begin{align}
    f &= \sum_{m,n=0}^\infty c_{m,n}\Phi_m \otimes \Psi'_n, \quad \text{in} \,\, X_+' \label{expansion for X-}\\
    g &= \sum_{m,n=0}^\infty d_{m,n} \Phi'_m \otimes \Psi_n,\quad \text{in} \,\, X_
    -' \label{expansion for X+}.
\end{align}
where $\overline{c_{m,n}} = \langle \Phi'_m \otimes \Psi_n | f\rangle$ and $\overline{d_{m,n}}=\langle \Phi_m \otimes \Psi'_n | g\rangle$, and
 the convergence of each series is in the sense of the weak * topology
 on $X_+'$ for (\ref{expansion for X-}) and $X_-'$ for (\ref{expansion for X+}). 
 \end{proposition}
 \begin{proof}
We show the first equality.
Due to Lemma~\ref{lemma4-6}, $\Phi'_m \otimes \Psi_n \in X_-'$.
Thus, $\langle \Phi'_m \otimes \Psi_n | f\rangle$ is well defined for $f\in X_-$ and $\overline{c_{m,n}} = \langle \Phi'_m \otimes \Psi_n | f\rangle$ due to Proposition~\ref{prop: 04261347}(2) if the right hand side converges.
For the convergence in $X'_{+}$ with respect to the weak * topology, we show that $\langle \sum_{m,n=0}^N c_{m,n} \Phi_m \otimes \Psi_n' \,|\, \tilde{f} \rangle$ converges as $N \rightarrow \infty$ for any $\tilde{f} \in X_+$. 
Let
\begin{eqnarray*}
\tilde{f} = \sum^\infty_{j'=0} q_{j'} \Phi'_{j'} \otimes \sum^{M'}_{i'=0} p_{i'} \Psi_{i'}
\end{eqnarray*}
where Lemma~\ref{lemma4-2} is applied for an element in $L^2 (\Sigma _-)$.
Since $f \in P(\Sigma_+) \otimes L^2(\Sigma_-)$, there is an integer $M = \max\{i \,|\, c_{i,j}\neq 0\}$.
Then, as $N\rightarrow \infty$, we have
\begin{eqnarray*}
\big\langle \sum_{i,j=0}^N c_{i,j}\Phi_i \otimes \Psi'_j \,|\, \tilde{f} \big\rangle 
&\rightarrow& \sum^{\infty}_{j,j' = 0} \sum^{M}_{i=0}\sum^{M'}_{i'=0}  c_{i,j}\, \overline{p_{i'}} \overline{q_{j'}}
\langle \Phi_i \otimes \Psi'_j  \,|\, \Phi'_{j'} \otimes \Psi_{i'}  \rangle \\
&=& \sum^M_{i=0}\sum^{M'}_{i'=0} c_{i,i'}\overline{q_{i}}\overline{p_{i'}} <\infty.
\end{eqnarray*}
\end{proof}


\subsection{Analytic continuation of the resolvent of the Perron-Frobenius operator 1}

Define the operators $Q_0, Q_1$ and $V_L (\varepsilon)$ on $L^2(\Sigma ) \simeq L^2(\Sigma_+ ) \otimes L^2(\Sigma_- )$ by 
\begin{eqnarray}
& & Q_0 = V_{L+} \otimes U_{R-}, \\
& & Q_1 = V_{L+}(-1)^{\omega _1} \otimes (-1)^{\omega _0} U_{R-},\\
& & V_L(\varepsilon ) = Q_0 + \varepsilon Q_1,
\end{eqnarray}
where a number $\varepsilon \geq 0$ is introduced because a perturbation theory of linear operators will be used later.
Our purpose in this subsection is to calculate the analytic continuation of the resolvent of $V(0) = Q_0$ based on the Gelfand triplet.
\begin{lemma}\label{lemma4-8}
For $f\in L^2 (\Sigma )$, $Q_0$ and $Q_1$ are given by
\begin{eqnarray*}
(Q_0f)(\omega ) &=& \frac{1}{2}f(\cdots ,\omega _{-1}.\, 0, \omega _1, \cdots )
   + \frac{1}{2}f(\cdots ,\omega _{-1}.\, 1, \omega _1, \cdots ), \\
(Q_1f)(\omega ) &=& \frac{1}{2}(-1)^{\omega _0}f(\cdots ,\omega _{-1}.\, 0, \omega _1, \cdots )
   - \frac{1}{2}(-1)^{\omega _0}f(\cdots ,\omega _{-1}.\, 1, \omega _1, \cdots ).
\end{eqnarray*}
In particular, we have $V_L(1) = Q_0 + Q_1 = V_L$.
\end{lemma}

\begin{proof}
We prove for an element $f(\omega ) =f_+(\omega_1, \omega _2, \cdots ) \otimes f_-(\cdots , \omega _{-1}, \omega _0)$.
\begin{eqnarray*}
& & (Q_0f)(\omega ) \\ 
&=& (V_{L+}f_+)(\omega_1, \omega _2, \cdots ) \otimes (U_{R-}f_-) (\cdots , \omega _{-1}, \omega _0) \\
&=& \bigl( 2^{-1} f_+(0, \omega _1,\cdots ) + 2^{-1}f_+(1, \omega _1,\cdots ) \bigr)
     \otimes f_- (\cdots , \omega _{-2}, \omega _{-1}) \\
&=& 2^{-1} f(\cdots , \omega _{-1}.\, 0, \omega _1, \cdots ) + 2^{-1}f(\cdots ,\omega _{-1}.\, 1, \omega _1, \cdots ).
\end{eqnarray*}
\begin{eqnarray*}
& &  (Q_1f)(\omega ) \\
&=& (V_{L+}(-1)^{\omega _1}f_+)(\omega _1, \omega _2,\cdots ) \otimes ((-1)^{\omega _0} U_{R-}f_-)(\cdots ,\omega _{-1}, \omega _0) \\
&=& V_{L+}((-1)^{\omega _1}f_+(\omega _1, \omega _2,\cdots )) \otimes (-1)^{\omega _0} f_-(\cdots ,\omega _{-2}, \omega _{-1}) \\
&=& 2^{-1} \left( (-1)^{0}f_+(0, \omega _1,\cdots ) + (-1)^{1} f_+(1, \omega _1, \cdots ) \right)
    \otimes (-1)^{\omega _0} f_-(\cdots ,\omega _{-2}, \omega _{-1}) \\
&=& 2^{-1}(-1)^{\omega _0}f(\cdots , \omega _{-1}.\, 0, \omega _1,\cdots ) - 2^{-1} (-1)^{\omega _0}f(\cdots , \omega _{-1}.\, 1, \omega _1,\cdots ).
\end{eqnarray*}
They prove the first two equalities.
For the last one, suppose that $\omega_0 = 0$.
Then, 
\begin{eqnarray*}
& & (Q_0 + Q_1)(f) \\
&=& 2^{-1} f(\cdots , \omega _{-1}.\, 0, \omega _1, \cdots ) + 2^{-1}f(\cdots ,\omega _{-1}.\, 1, \omega _1, \cdots ) \\
& & +2^{-1}f(\cdots , \omega _{-1}.\, 0, \omega _1,\cdots ) - 2^{-1} f(\cdots , \omega _{-1}.\, 1, \omega _1,\cdots ) \\
&=& f(\cdots ,\omega _{-1}.\, 0, \omega _1, \cdots ).
\end{eqnarray*}
This coincides with $V_Lf$ for $\omega_0 = 0$ given by (\ref{PF_two_side}).
The case $\omega_0 = 1$ is proved in the same way.
\end{proof}
\begin{lemma}\label{lemma4-9}\,
\\
(1) The adjoint operators of $Q_0$ and $Q_1$ are given by
\begin{equation}
Q_0^* = U_{L+} \otimes V_{R-}, \quad Q_1^* = (-1)^{\omega _1}U_{L+} \otimes V_{R-}(-1)^{\omega _0}.
\end{equation}
Further, they are continuous operators on $X_+= L^2(\Sigma_+) \otimes P(\Sigma _-)$.
Thus, they induce the dual operators $Q_0^\times$ and $Q_1^\times$ on $X_+'$.
\\
(2) When $|\lambda|$ is sufficiently large, 
$(\overline{\lambda} - Q^*_0)^{-1}$ is a continuous operator on $X_+$ and it induces the dual operator $((\lambda - Q_0)^{-1})^\times$ on $X_+'$, which is denoted by $(\lambda - Q_0)^{-1}$ for simplicity.
\end{lemma}

\begin{proof}
(1) For $f = f_+ \otimes f_-$ and $g = g_+\otimes g_-$ in $L^2 (\Sigma ) \simeq L^2(\Sigma _+) \otimes L^2(\Sigma _-)$,
\begin{eqnarray*}
(f, Q_0^*g)_{L^2(\Sigma )}
&=& (Q_0f, g)_{L^2(\Sigma )} \\
&=& \left( V_{L+}f_+ \otimes U_{R-}f_-,\, g_+ \otimes g_- \right)_{L^2(\Sigma )} \\
&=& (V_{L+}f_+, g_+)_{L^2(\Sigma_+ )} \cdot (U_{R-}f_-, g_-)_{L^2(\Sigma_- )} \\
&=& (f_+, V_{L+}^*g_+)_{L^2(\Sigma_+ )} \cdot (f_-, U_{R-}^* g_-)_{L^2(\Sigma_- )} \\
&=& (f_+, U_{L+}g_+)_{L^2(\Sigma_+ )} \cdot (f_-, V_{R-} g_-)_{L^2(\Sigma_- )} \\
&=& (f_+ \otimes f_-, (U_{L+}g_+) \otimes (V_{R-} g_-) )_{L^2(\Sigma )} \\
&=& (f, (U_{L+} \otimes V_{R-}) g)_{L^2(\Sigma )}.
\end{eqnarray*}
This shows the first equality.
Next, suppose that $g = g_+\otimes g_- \in X_+$, where $g_+\in L^2(\Sigma _+)$ and $ g_-\in P(\Sigma _-)$;
\begin{eqnarray*}
Q_0^*g = (U_{L+}g_+) \otimes (V_{R-}g_-)
\end{eqnarray*}
Obviously $U_{L+}$ is continuous on $L^2(\Sigma _+)$, and $V_{R-}$ is continuous on $P(\Sigma _-)$ 
that was observed from Prop.~\ref{prop: behavior of h via PF op} 
(for $V_{L+}$ on $P(\Sigma_+)$).
The statement for $Q_1^*$ is proved in a similar manner.

(2) Since $P_M(\Sigma_-)$ is finite dimensional, the space $L^2 (\Sigma_+ )\otimes P_M(\Sigma_-)$ is a Banach space, on which $Q^*_0$ is continuous for any $M\geq 0$.
The spectrum set of the continuous operator $Q^*_0$ on $L^2 (\Sigma_+ )\otimes P_M(\Sigma_-)$ is compact and its diameter is independent of $M$.
Hence, if $|\lambda|$ is sufficiently large, $\lambda$ is a point on the resolvent set and $(\overline{\lambda} - Q^*_0)^{-1}$ is a continuous bijection on $L^2 (\Sigma_+) \otimes P_M(\Sigma_-)$, and so is on its inductive limit $X_+$.
\end{proof}

Though our purpose is to calculate the resolvent of $V_L = Q_0 + Q_1$,
at first, let us derive the resolvent and its analytic continuation of $Q_0 = V_L(0)$
because $Q_1$ will be considered as a perturbation.

Suppose that $f, g\in L^2(\Sigma)$ and $|\lambda|$ is sufficiently large.
Then, the Neumann series yields
\begin{eqnarray*}
\left( (\lambda -Q_0)^{-1}f, g\right)_{L^2(\Sigma )}
 &=& \frac{1}{\lambda }\sum^\infty_{k=0}\frac{1}{\lambda ^k} \left( Q_0^k f, g \right)_{L^2(\Sigma )} \\
 &=& \frac{1}{\lambda }\sum^\infty_{k=0}\frac{1}{\lambda ^k} \left( f, (Q_0^*)^k g \right)_{L^2(\Sigma )}.
\end{eqnarray*}
Now suppose that $f\in X_-$ and $g\in X_+$.
Since $Q_0^*$ is continuous on $X_+$ due to the previous lemma, the dual operator $Q_0^\times : X_+' \to X_+'$ is defined and continuous (see Section 2).
By using the triplet $X_- \subset L^2(\Sigma) \subset X_+'$~(\ref{triplet -+}), 
the right hand side above is rewritten as
\begin{eqnarray*}
\left( (\lambda -Q_0)^{-1}f, g\right)_{L^2(\Sigma )}
 =\frac{1}{\lambda }\sum^\infty_{k=0}\frac{1}{\lambda ^k} \bigl\langle (Q_0^\times )^kf \, |\,  g  \bigr\rangle.
\end{eqnarray*}
Now we expand $f\in X_-$ using (\ref{expansion for X-}) as
\begin{eqnarray*}
 (\lambda -Q_0)^{-1}f 
= \frac{1}{\lambda }\sum^\infty_{k=0}\frac{1}{\lambda ^k} (Q_0^\times )^k \sum_{m,n=0}^\infty c_{m,n}\Phi_m \otimes \Psi'_n, \quad \text{in}\,\,  X_+'.
\end{eqnarray*}
Since $Q_0^\times : X_+' \to X_+'$ is a natural lift of the action of $Q_0$ on $X_+$,
we have
\begin{eqnarray*}
Q_0^\times [\Phi_m \otimes \Psi'_n] &=& V_{L+} \Phi_m \otimes U_{R-}^\times \Psi'_n \\
&=& 1/2^{m+n} \cdot  \Phi_m \otimes \Psi'_n,
\end{eqnarray*}
with the aid of Lemma~\ref{lem: eigen function formula}.
Thus, we obtain
\begin{eqnarray}
 (\lambda -Q_0)^{-1}f
&=& \frac{1}{\lambda }\sum_{m,n=0}^\infty \sum^\infty_{k=0}\frac{c_{m,n}}{(\lambda \cdot 2^{m+n})^k}\Phi_m \otimes \Psi'_n \nonumber \\
&=& \sum_{m,n=0}^\infty \frac{c_{m,n}}{\lambda - 2^{-(m+n)}}\Phi_m \otimes \Psi'_n,
\quad \text{in}\,\,  X_+'.
\label{4-20}
\end{eqnarray}
This gives the analytic continuation of the resolvent $(\lambda - Q_0)^{-1}$ to the inside of the unit circle as an operator from $X_-$ into $X_+'$ based on the triplet (\ref{triplet -+}). 
It is called the generalized resolvent.
\begin{theorem}\label{thm4-10}
The generalized eigenvalues of the operator $Q_0$ with respect to the Gelfand triplet  
$X_- \subset L^2(\Sigma) \subset X_+'$ are given by $2^{-(m+n)}\,\, (m,n \geq 0)$.
Their geometric and algebraic multiplicities are $m+n+1$
and the set $\{ \Phi_i \otimes \Psi'_j \}_{i+j = m+n} \subset X_+'$ spans the generalized eigenspace.
\end{theorem}

\begin{remark}\label{remark4-11}
A similar statement holds for the triplet $X_+ \subset L^2(\Sigma) \subset X_-'$, for which generalized eigenvalues are common and the
generalized eigenvectors are replaced by $\{ \Phi'_i \otimes \Psi_j \}_{i+j = m+n}$; for $g\in X_+$ expanded by (\ref{expansion for X+}), we have
\begin{eqnarray}
 (\lambda -Q_0^* )^{-1}g
&=& \sum_{m,n=0}^\infty \frac{d_{m,n}}{\lambda - 2^{-(m+n)}} \Phi'_m \otimes \Psi_n ,
\quad \text{in}\,\,  X_-'.
\label{4-20b}
\end{eqnarray}
\end{remark}


\subsection{Analytic continuation of the resolvent of the Perron-Frobenius operator 2}

In this subsection, the resolvent of the operator $V_L(\varepsilon ) = Q_0 + \varepsilon Q_1$ for $\varepsilon \neq 0$ will be given. In particular, the analytic continuation of the resolvent of the Perron-Frobenius operator of the left shift will be obtained for $\varepsilon= 1$.

When $|\lambda |$ is sufficiently large, the Neumann series yields
\begin{eqnarray*}
(\lambda -V_L(\varepsilon ))^{-1}
&=& (\lambda -Q_0)^{-1} \circ (\mathrm{Id} - \varepsilon Q_1 \circ (\lambda -Q_0)^{-1})^{-1} \\
&=& (\lambda -Q_0)^{-1}\sum^\infty_{k=0} \varepsilon ^k \left( Q_1\circ (\lambda -Q_0)^{-1} \right)^k.
\end{eqnarray*}
For any $f,g\in L^2 (\Sigma )$, we have
\begin{eqnarray*}
\left( (\lambda -V_L(\varepsilon ))^{-1}f, g \right)_{L^2(\Sigma )}
&=& \sum^\infty_{k=0} \varepsilon ^k \left( (\lambda -Q_0)^{-1} \circ \bigl( Q_1\circ (\lambda -Q_0)^{-1} \bigr)^kf, g \right)_{L^2(\Sigma )} \\
&=&  \sum^\infty_{k=0} \varepsilon ^k A_k(\lambda),
\end{eqnarray*}
where the coefficient of $\varepsilon^k$ is denoted by $A_k(\lambda)$.

Now assume that $f\in X_-$ and $g\in X_+$, and calculate the analytic continuation of $A_k(\lambda)$
\begin{eqnarray*}
 A_k(\lambda )
&=& \Bigl\langle (\lambda -Q_0)^{-1} \circ \bigl( Q_1^\times \circ (\lambda -Q_0)^{-1} \bigr)^kf\, \Big| \, g \Bigr\rangle
\end{eqnarray*}
by regarding $(\lambda -Q_0)^{-1} \circ \bigl( Q_1^\times \circ (\lambda -Q_0)^{-1} \bigr)^kf$ as an element of $X_+'$, where $(\lambda - Q_0)^{-1}$ is the dual operator on $X_+'$ (see Lemma \ref{lemma4-9} (2)).
\\
\\
\textbf{(i)} $\bm{k=0 :}$ In this case, (\ref{4-20}) gives
\begin{eqnarray*}
A_0(\lambda )
= \sum^\infty_{m,n = 0} \frac{c_{m,n}}{\lambda - 2^{-(m+n)}} 
\Bigl\langle \Phi_m\otimes \Psi'_n \,|\, g \Bigr\rangle.
\end{eqnarray*}
Here, recall that $\Phi_m\otimes \Psi'_n\in X_+'$.
\\
\\
\noindent \textbf{(ii)} $\bm{k=1 :}$  $A_1(\lambda)$ is given by
\begin{eqnarray*}
A_1(\lambda)
&=& \sum^\infty_{m,n}\frac{c_{m,n}}{\lambda - 2^{-(m+n)}} 
     \Bigl\langle (\lambda - Q_0)^{-1} \circ Q_1^\times \Phi_m\otimes \Psi'_n \,|\, g \Bigr\rangle \\
&=&  \sum^\infty_{m,n}\frac{c_{m,n}}{\lambda - 2^{-(m+n)}} 
     \Bigl\langle Q_1^\times \Phi_m\otimes \Psi'_n \,|\, (\overline{\lambda} - Q^*_0)^{-1} g \Bigr\rangle.
\end{eqnarray*}

 Lemma~\ref{lemma4-9} shows $(\overline{\lambda} - Q^*_0)^{-1} g \in X_+$.
Thus, the right hand side above is well-defined when $Q_1^\times \Phi_m\otimes \Psi'_n \in X_+'$. The properties of $Q_1^\times \Phi_m\otimes \Psi'_n$ are summarized as follows:

\begin{lemma}\label{lemma4-12}$\,$
\\
(1) $Q_1^\times \Phi_m\otimes \Psi'_n \in  P(\Sigma _+) \otimes P(\Sigma _-)' \subset X_+'$. 
\\
(2) $\langle Q_1^\times \Phi_m\otimes \Psi'_n  \,|\, \Phi'_{m'} \otimes \Psi_{n'}  \rangle
 = \langle \Psi'_n \,|\, V_{L-}(-1)^{\omega _0}\Psi_{n'}  \rangle \cdot
\overline{\langle \Phi'_{m'} \,|\, V_{L+}(-1)^{\omega _1}\Phi_{m}  \rangle}$
\\
(3) $\langle Q_1^\times \Phi_m\otimes \Psi'_n  \,|\, \Phi'_{m'} \otimes \Psi_{n'}  \rangle = 0$
either $m' \geq m$ or $n' \leq n$.
\\
(4) $Q_1^\times \Phi_m\otimes \Psi'_n $
is expanded as
\begin{eqnarray*}
Q_1^\times \Phi_m\otimes \Psi'_n  
 = \sum^\infty_{i,j=0} \langle Q_1^\times \Phi_m\otimes \Psi'_n  \,|\, \Phi'_{i} \otimes \Psi_{j} \rangle 
\cdot \Phi_i \otimes \Psi'_j
\end{eqnarray*}
with respect to the weak * topology on $X'_+$ (the right hand side is a finite sum because of the statement (3) and $\Psi_j\in P(\Sigma_-)$).
\end{lemma}

\begin{proof}
(1) By a similar calculation to the proof of Lemma~\ref{lemma4-8}, 
\begin{eqnarray*}
Q_1^\times \Phi_m\otimes \Psi'_n
&=& (V_{L+}(-1)^{\omega _1} \Phi_m)(\omega_1, \omega _2,\cdots )\otimes \bigl( (-1)^{\omega _0} U_{R-} \bigr)^\times \Psi'_n \\
&=& \frac{1}{2}\left( \Phi_m (0,\omega _1, \omega _2, \cdots ) - \Phi_m(1, \omega _1, \omega _2,\cdots ) \right)
\otimes\bigl( (-1)^{\omega _0} U_{R-} \bigr)^\times \Psi'_n.
\end{eqnarray*}
For the first factor, recall that $\Phi_m$ is a monic polynomial of $h$ of degree $m$.
In particular, the leading term (the highest degree term in $h$) of 
$\Phi_m (0,\omega _1, \omega _2, \cdots ) - \Phi_m(1, \omega _1, \omega _2,\cdots )$ is (see also Prop.\ref{prop: behavior of h via PF op})
\begin{eqnarray*}
& & h(0, \omega _1,\cdots )^m - h(1, \omega _1,\cdots )^m \\
& = & \left( \frac{0}{2} + \frac{\omega _1}{2^2} + \cdots  \right)^m
  - \left( \frac{1}{2} + \frac{\omega _1}{2^2} + \cdots  \right)^m \\
&=& \frac{1}{2^m}h^m - \frac{1}{2^m} \left( 1 + h \right)^m.
\end{eqnarray*}
This is a polynomial of $h$ of degree $m-1$, which implies that
$V_{L+}(-1)^{\omega _1} \Phi_m \in P_{m-1}(\Sigma _+) \subset  P(\Sigma _+)$.
For the second factor, let us show that for any $f_- \in P(\Sigma_-)$, $\langle \bigl( (-1)^{\omega _0}  U_{R-} \bigr)^\times \Psi'_n 
      \,|\, f_- \rangle$ exists and continuous in $f_-$
 \[ \langle\bigl( (-1)^{\omega _0}  U_{R-} \bigr)^\times \Psi'_n \,|\, f_- \rangle
 = \langle  \Psi'_n \,|\, \bigl( (-1)^{\omega _0}  U_{R-} \bigr)^* f_- \rangle 
 = \langle  \Psi'_n \,|\, V_{R-}(-1)^{\omega _0}f_- \rangle. \]
 Since $f_{-}$ is a linear combination of $\Psi_m$'s, we can verify that $ V_{R-}(-1)^{\omega _0}f_- \in P(\Sigma_-)$ by the same way as above for the first factor $V_{L+}(-1)^{\omega _1} \Phi_m$.
 Since $V_{L+}$ is continuous, $\langle  \Psi'_n \,|\, V_{R-}(-1)^{\omega _0}f_- \rangle$ tends to zero as $f_- \to 0$ in $P(\Sigma_-)$.
This proves $\bigl( (-1)^{\omega _0}  U_{R-} \bigr)^\times \Psi'_n \in  P(\Sigma _-)'$.

(2) Due to the statement (1), the bilinear form $\langle Q_1^\times \Phi_m\otimes \Psi'_n  \,|\, \Phi'_{m'} \otimes \Psi_{n'}  \rangle$ is well-defined. 
It is calculated by the definition as
\begin{eqnarray*}
& & \langle Q_1^\times \Phi_m\otimes \Psi'_n  \,|\, \Phi'_{m'} \otimes \Psi_{n'}  \rangle \\
&=& \langle V_{L+}(-1)^{\omega _1} \Phi_m \otimes \bigl( (-1)^{\omega _0}  U_{R-} \bigr)^\times \Psi'_n  \,|\, 
       \Phi'_{m'} \otimes \Psi_{n'}  \rangle \\
&=& \langle \bigl( (-1)^{\omega _0}U_{R-}\bigr)^\times \Psi'_n \,|\, \Psi_{n'} \rangle \cdot 
       \overline{\langle \Phi'_{m'} \,|\,V_{L+}(-1)^{\omega _1} \Phi_m \rangle} \\
&=& \langle \Psi'_n \,|\, U^*_{R-}(-1)^{\omega _0} \Psi_{n'} \rangle \cdot 
       \overline{\langle \Phi'_{m'} \,|\,V_{L+}(-1)^{\omega _1} \Phi_m \rangle} \\
&=& \langle \Psi'_n \,|\, V_{R-}(-1)^{\omega _0} \Psi_{n'} \rangle \cdot 
       \overline{\langle \Phi'_{m'} \,|\,V_{L+}(-1)^{\omega _1} \Phi_m \rangle}.
\end{eqnarray*}

(3) From the proof of (1), we found that $V_{L+}(-1)^{\omega _1} \Phi_m \in P_{m-1}(\Sigma _+)$
is a polynomial of $h$ of degree $m-1$.
Hence, it is a linear combination of $\Phi_0, \Phi_1, \cdots ,\Phi_{m-1}$.
Since $\langle \Phi'_m \,|\, \Phi_n \rangle = \delta _{m,n}$, 
it turns out that $\langle \Phi'_{m'} \,|\,V_{L+}(-1)^{\omega _1} \Phi_m \rangle= 0$
when $m' \geq m$.
Similarly, $\langle \Psi'_n \,|\, V_{R-}(-1)^{\omega _0} \Psi_{n'} \rangle = 0$
when $n' \leq n$.

The proof of (4) is the same as that of Prop.~\ref{prop4-7}.
\end{proof}$\,$
\\
\textbf{(iii)} $\bm{k\geq 2 :}$ In this case, 
\begin{eqnarray*}
A_k(\lambda)
= \sum^\infty_{m,n}\frac{c_{m,n}}{\lambda - 2^{-(m+n)}} 
 \Bigl\langle \bigl( Q_1^\times \circ (\lambda - Q_0)^{-1}\bigr)^{k-1} Q_1^\times \Phi_m\otimes \Psi'_n \,|\, (\overline{\lambda} - Q^*_0)^{-1} g \Bigr\rangle .
\end{eqnarray*}
By Lemma~\ref{lemma4-12}(4), $Q_1^\times \Phi_m\otimes \Psi'_n$ is expanded as
\begin{eqnarray*}
Q_1^\times \Phi_m\otimes \Psi'_n  
 = \sum^\infty_{i_1,j_1} \langle Q_1^\times \Phi_m\otimes \Psi'_n  \,|\, \Phi'_{i_1} \otimes \Psi_{j_1} \rangle 
\cdot \Phi_{i_1} \otimes \Psi'_{j_1}.
\end{eqnarray*}
Substituting it into the above $A_k(\lambda)$ gives
\begin{eqnarray*}
A_k(\lambda ) &=&
\sum^\infty_{m,n}\sum^\infty_{i_1, j_1} 
\frac{c_{m,n}}{\lambda - 2^{-(m+n)}} \langle Q_1^\times \Phi_m\otimes \Psi'_n  \,|\, \Phi'_{i_1} \otimes \Psi_{j_1} \rangle  \\
& & \quad \times \Bigl\langle (Q_1^\times \circ (\lambda - Q_0)^{-1})^{k-1} \Phi_{i_1}\otimes \Psi'_{j_1} \,|\, (\overline{\lambda} - Q^*_0)^{-1} g \Bigr\rangle.
\end{eqnarray*}
Hence, (\ref{4-20}) yields
\begin{eqnarray*}
A_k(\lambda ) &=&
\sum^\infty_{m,n}\sum^\infty_{i_1, j_1} 
\frac{c_{m,n}}{\lambda - 2^{-(m+n)}} \frac{1}{\lambda -2^{-(i_1 + j_1)}} 
     \langle Q_1^\times \Phi_m\otimes \Psi'_n  \,|\, \Phi'_{i_1} \otimes \Psi_{j_1} \rangle  \\
& & \quad \times \Bigl\langle (Q_1^\times \circ (\lambda - Q_0)^{-1})^{k-2} Q_1^\times \Phi_{i_1}\otimes \Psi'_{j_1} \,|\, (\overline{\lambda} - Q^*_0)^{-1} g \Bigr\rangle .
\end{eqnarray*}
Repeating this procedure, we obtain
\begin{eqnarray*}
& & A_k(\lambda ) \\
&=& 
\sum^\infty_{m,n}\sum^\infty_{i_1, j_1} \cdots \sum^\infty_{i_{k-1}, j_{k-1}} 
\left( \frac{c_{m,n}}{\lambda - 2^{-(m+n)}}
\frac{1}{\lambda -2^{-(i_1 + j_1)}}\cdots \frac{1}{\lambda -2^{-(i_{k-1} + j_{k-1})}} \right) \\
& & \qquad \times \langle Q_1^\times \Phi_m\otimes \Psi'_n  \,|\, \Phi'_{i_1} \otimes \Psi_{j_1} \rangle
 \langle Q_1^\times \Phi_{i_1}\otimes \Psi'_{j_1}  \,|\, \Phi'_{i_2} \otimes \Psi_{j_2} \rangle
\times \cdots \\
& & \qquad\quad \times \langle Q_1^\times \Phi_{i_{k-1}}\otimes \Psi'_{j_{k-1}}  \,|\,  (\overline{\lambda} - Q^*_0)^{-1} g \rangle.
\end{eqnarray*}
To summarize (ii) and (iii) with \eqref{4-20b}, for any $k \ge 1$, we have
\begin{eqnarray*}
& & A_k(\lambda ) \\
&=& 
\sum^\infty_{m,n}\sum^\infty_{i_1, j_1} \cdots \sum^\infty_{i_{k-1}, j_{k-1}} \sum^\infty_{i_k,j_k} \\
& & 
\left( \frac{c_{m,n}}{\lambda - 2^{-(m+n)}}
\frac{1}{\lambda -2^{-(i_1 + j_1)}}\cdots \frac{1}{\lambda -2^{-(i_{k-1} + j_{k-1})}}\frac{\overline{d_{i_k,j_k}}}{\lambda - 2^{-(i_k+j_k)}} \right) \\
& & \qquad \times \langle Q_1^\times \Phi_{m}\otimes \Psi'_{n}  \,|\, \Phi'_{i_1} \otimes \Psi_{j_1} \rangle
 \langle Q_1^\times \Phi_{i_1}\otimes \Psi'_{j_1}  \,|\, \Phi'_{i_2} \otimes \Psi_{j_2} \rangle
\times \cdots \\
& & \qquad\quad \times \langle Q_1^\times \Phi_{i_{k-1}}\otimes \Psi'_{j_{k-1}}  \,|\,  \Phi'_{i_k} \otimes \Psi_{j_k} \rangle.
\end{eqnarray*}
Lemma~{\ref{lemma4-12}} (3) simplifies it as
\begin{eqnarray}
& & A_k(\lambda ) \nonumber \\
&=& \sum_{\Omega } 
\frac{c_{m,n}}{\lambda - 2^{-(m+n)}}\frac{1}{\lambda -2^{-(i_1 + j_1)}}
  \cdots \frac{1}{\lambda -2^{-(i_{k-1} + j_{k-1})}}\frac{\overline{d_{i_k,j_k}}}{\lambda - 2^{-(i_k + j_k)}} \nonumber \\
&\times & \langle Q_1^\times \Phi_m\otimes \Psi'_n  \,|\, \Phi'_{i_1} \otimes \Psi_{j_1} \rangle \times \cdots \times \langle Q_1^\times \Phi_{i_{k-1}}\otimes \Psi'_{j_{k-1}}  \,|\,  \Phi'_{i_k} \otimes \Psi_{j_k} \rangle, \nonumber \\
\label{Ak}
\end{eqnarray}
where the set $\Omega $ is defined by
\begin{eqnarray*}
\Omega := \{ 0\leq i_k<i_{k-1}< \cdots < i_1 < m,\,\, 0\leq n<j_{1}< \cdots < j_{k-1} < j_k\}.
\end{eqnarray*}
In particular, the inequalities $j_l \geq l$ and $i_l \geq k-l$ hold,
which implies $i_l + j_l \geq k$ ($m+n \geq k$ also holds).
This shows that $A_k(\lambda )$ has no singularities $\lambda =1, 2^{-1} ,\cdots , 2^{-(k-1)}$, however, $\lambda =2^{-k}$ may be a pole of $A_k(\lambda)$ of order $k+1$ because (\ref{Ak}) includes $k+1$ factors of the form $\lambda - 2^{-(i+j)}$ in the denominator.

\begin{lemma}
$A_k(\lambda)$ has a pole $\lambda = 2^{-k}$ of order $k+1$.
\end{lemma}
\begin{proof}
In Lemma~\ref{lemma4-12} (2), it is shown that
\[ \langle Q_1^\times \Phi_m\otimes \Psi'_n  \,|\, \Phi'_{i_1} \otimes \Psi_{j_1}  \rangle
 = \langle \Psi'_n \,|\, V_{L-}(-1)^{\omega _0}\Psi_{j_1}  \rangle \cdot
\overline{\langle \Phi'_{i_1} \,|\, V_{L+}(-1)^{\omega _1}\Phi_{m}  \rangle}. \]
In the proof of Lemma~\ref{lemma4-12} (1), we proved that $V_{L+}(-1)^{\omega _1}\Phi_{m}$ is a linear combination of $\Phi_0, \cdots, \Phi_{m-1}$ (recall that $h^{m-1}$ is a linear combination of $\Phi_0, \cdots, \Phi_{m-1}$).
From its proof, it is easy to confirm that the coefficient of $ \Phi_{m-1}$ is not zero.
Therefore, $\langle \Phi'_{i_1} \,|\, V_{L+}(-1)^{\omega _1}\Phi_{m}  \rangle \neq 0$ when $i_1 = m-1$.

Now we put $(n, j_1) = (k-m, k-m+1)$.
By the same reason as above, $ \langle \Psi'_n \,|\, V_{L-}(-1)^{\omega _0}\Psi_{j_1}  \rangle \neq 0$, and we have $m+n = i_1 + j_1 = k$.
This proves that in the set $\Omega$ of subscripts, there is a tuple $(m,n, i_1, j_1, \cdots)$ such that $\langle Q_1^\times \Phi_m\otimes \Psi'_n  \,|\, \Phi'_{i_1} \otimes \Psi_{j_1}  \rangle \neq 0$ and $m+n = i_1 + j_1 = k$.
By the same argument to the other factors in (\ref{Ak}), we find out that $A_k(\lambda )$ includes a term $1/(\lambda - 2^{-k})^{k+1} \times \text{(nonzero constant)}$.
\end{proof}

We are now in a position to show our main theorem.

\begin{theorem}\label{main thm} For any $0\leq \varepsilon \leq 1$,
\\
(1) The resolvent $(\lambda -V_L(\varepsilon ))^{-1}$ has an analytic continuation
to the inside of the unit circle as an operator from $X_-$ into $X_+'$.
The generalized eigenvalues with respect to the Gelfand triplet $X_- \subset L^2(\Sigma )\subset X_+'$ are $2^{-k},\,\, k=0,1,2,\cdots $.
\\
(2) The generalized eigenspace of the generalized eigenvalue $2^{-k}$ is a $k+1$ dimensional space.
When $\varepsilon \neq 0$,
\begin{eqnarray*}
(\text{algebraic multiplicity}) = k+1, \quad
(\text{geometric multiplicity}) = 1.
\end{eqnarray*}
(3) The statements (1) and (2) hold for the Perron-Frobenius operator $V_L$ of the two-sided shift.
\end{theorem}

\begin{proof}
(1) has been already obtained. If (2) is true, (3) is trivial because $V_L = V_L(1)$. Let us prove the statement (2).

Let $W_k (\varepsilon ) \subset X_+'$ be the generalized eigenspace associated with the generalized eigenvalue
$\lambda = 2^{-k}$.
This is the image of the generalized Riesz projection (Section 2)
\begin{eqnarray*}
W_k(\varepsilon ) = \frac{1}{2\pi i} \mathrm{Im} \oint_{C_k} (z-V_L(\varepsilon ))^{-1} dz,
\end{eqnarray*}
where $C_k$ is a positively oriented closed curve, which encircles an eigenvalue $2^{-k}$
but no other eigenvalues.
When $\varepsilon =0$ (i.e. $V_L(\varepsilon ) = Q_0$), $W_k(0)$ is a $k+1$ dimensional space 
given by $W_k(0) = \mathrm{span} \{ \Phi_i \otimes \Psi'_j \}_{i+j = k}$
as was shown in Theorem~\ref{thm4-10}.
Since the generalized eigenvalues are independent of $\varepsilon $, there are no eigenvalues which cross the integral pass $C_k$ as $\varepsilon$ varies.
This means that $\mathrm{dim}\, W_k(\varepsilon ) = k+1 \,(= \text{the algebraic multiplicity})$ even when $\varepsilon \neq 0$.
Since $\lambda =2^{-k}$ is a pole of $(\lambda -V_L(\varepsilon ))^{-1}$ of order $k+1$,
the geometric multiplicity is one.
\end{proof}


\textbf{\large{Acknowledgments.}}
This work is supported by a JST CREST Grant, Number JPMJCR1913, Japan, 
JST ACTX Grant, Number JPMJAX2004, Japan, 
and 2019 IMI Joint Use Research Program.


\end{document}